\documentclass[reqno]{amsart}
\usepackage{amssymb,amsmath,amsfonts,amscd,amsthm,pb-diagram}
\usepackage{amsbsy,bm}
\usepackage[usenames,dvipsnames]{color}
\usepackage[normalem]{ulem}
\usepackage{needspace}

\newtheorem{theorem}{Theorem}[section]

\newtheorem{lemma}{Lemma}[section]
\newtheorem{proposition}{Proposition}[section]

\newtheorem{remark}{Remark}[section]

\newtheorem{example}{Example}[section]
\newtheorem*{problem}{Problem}
\newtheorem*{theoremA}{Theorem A}
\newtheorem*{theoremB}{Theorem B}

\begin{document}
\title[Legendrian submanifolds with conformal Maslov form]
{Legendrian submanifolds in the unit sphere with conformal Maslov form and
constant sectional curvature}

\author{Yong Luo}
\address{
Mathematical Science Research Center,
Chongqing University of Technology, Chongqing 400054,
People's Republic of China}
\email{yongluo-math@cqut.edu.cn}

\author{Cheng Xing}
\address{
School of Mathematics and Statistics,
Henan Normal University, Xinxiang 453007,
People's Republic of China}
\email{xingcheng@nankai.edu.cn}


\thanks{The project was supported by National Natural Science Foundation
of China (Grant Nos. 12671062 and 12501063), Postdoctoral Fellowship Program
of CPSF (Grant No. GZC20252036), China Postdoctoral Science Foundation (Grant
No. 2025M783098), and Natural Science Foundation of Henan Province (Grant Nos.
252300421303 and 262300421852)}

\thanks{C. Xing is the corresponding author}

\date{}


\begin{abstract}
This paper is concerned with the study on Legendrian submanifolds with
conformal Maslov form in the unit sphere $\mathbb{S}^{2n+1}$, which admits
a Sasakian structure $(\varphi,\xi,\eta,g)$ for $n\ge2$. As the main result,
we classify such submanifolds with constant sectional curvature, motivated by
the classification result of the minimal Legendrian submanifolds with constant
sectional curvature. Moreover, we prove that, for a closed Legendrian submanifold
$M^n$ in $\mathbb{S}^{2n+1}$ with conformal Maslov form, if its sectional curvature
satisfies the pinching $0\leq\sec_g\leq1$, then either $M^n$ is the totally geodesic
Legendrian sphere with $\sec_g=1$, or $M^n$ is a closed embedded weighted Clifford torus
with $\sec_g=0$. This extends the corresponding pinching theorem of Dillen--Vrancken
(J Math Pures Appl 69:85--93 1990) from minimal Legendrian submanifolds to the
conformal Maslov class under the same curvature bounds.
\end{abstract}

\maketitle
\numberwithin{equation}{section}

\section{Introduction}\label{sect:1}

Classification and pinching theorems for submanifolds in the unit sphere are two
classes of central and complementary themes in submanifold geometry. A classical
result of Cartan \cite{Car38} states that an isoparametric hypersurface with two
distinct principal curvatures is locally congruent to a product of two spheres.
On the rigidity side, Simons \cite{Sim68}, Lawson \cite{Law69} and Chern--do 
Carmo--Kobayashi \cite{CDK70} proved that a closed minimal submanifold $M^n\subset
\mathbb S^{n+p}$ satisfying $\|h\|^2\leq n/(2-1/p)$ is either totally geodesic or 
congruent to a minimal Clifford hypersurface or a Veronese surface, where $h$ is 
the second fundamental form and $p$ the codimension. Together, these results 
illustrate how geometric hypotheses and sharp curvature bounds can single out 
explicit submanifolds. For related classification and rigidity results in 
spheres, see \cite{AL15,Bre13,CCJ07,HLWZ20,Imm08,MN14} and the references therein.

As a real hypersurface of the complex Euclidean space $\mathbb C^{n+1}$, the
$(2n+1)$-dimensional unit sphere $\mathbb S^{2n+1}$ naturally carries a Sasakian
structure $(\varphi,\xi,\eta,g)$. An $m$-dimensional submanifold $M^m$ of $\mathbb
S^{2n+1}$ is said to be \textit{$C$-totally real}, or equivalently \textit{integral},
if the contact form vanishes on $M^m$, that is, $\eta|_{TM^m}=0$. When $m=n$, it
is called \textit{Legendrian}. This is the integral case of maximal dimension. The
Legendrian condition closely links intrinsic Riemannian geometry to the ambient contact
geometry and makes classification and rigidity problems particularly compelling. For
extensive studies, see {\cite{BBK95,But09,CD12,LW01,LLV20,Luo17,Luo18,LS22,LSY22,
Mih17,Sas14,SS20,WX26,YQ22} and the references therein.

The simplest Legendrian submanifold in $\mathbb S^{2n+1}$ is the totally geodesic
sphere, which is trivially minimal and has constant sectional curvature. This
naturally leads to the problem of finding and classifying minimal Legendrian
submanifolds that are not totally geodesic with constant sectional curvature.
A prominent example is the Clifford torus.

\begin{example}\label{exa:1.1}
Let $\mathbb S^1(r)$ be a circle of radius $r$ and denote the $n$-dimensional
torus by
\begin{equation}\label{eqn:1.1}
T^n=\mathbb S^1\big(\tfrac{1}{\sqrt{n+1}}\big)\times\cdots
\times\mathbb S^1\big(\tfrac{1}{\sqrt{n+1}}\big).
\end{equation}
Then the Clifford immersion $F_1(u):T^n\to\mathbb S^{2n+1}\subset\mathbb C^{n+1}$,
given by
\begin{equation}\label{eqn:1.2}
\begin{aligned}
F_1(u_1,\ldots,u_n)
=\tfrac{1}{\sqrt{n+1}}
(&\cos u_1,\sin u_1,\ldots,\cos u_n,\sin u_n,\\
&\cos(u_1+\cdots+u_n),-\sin(u_1+\cdots+u_n)),
\end{aligned}
\end{equation}
is a minimal Legendrian immersion with flat induced metric.
\end{example}

Combining the fundamental results of Yamaguchi--Kon--Ikawa \cite{YKI76} and
Cheng--He--Hu \cite{CHH21}, one obtains the following classification theorem
for minimal Legendrian submanifolds in the unit sphere $\mathbb S^{2n+1}$ with
constant sectional curvature.

\begin{theoremA}[\cite{CHH21,YKI76}]
Let $M^n$ $(n\geq2)$ be a minimal Legendrian submanifold in the unit sphere
$\mathbb S^{2n+1}$ with constant sectional curvature. Then $M^n$ is locally
congruent either to a totally geodesic Legendrian sphere or to the immersion
$F_1$ given in Example \ref{exa:1.1}.
\end{theoremA}

A natural generalization of submanifolds with constant sectional curvature
is given by submanifolds with parallel Ricci tensor, which include Einstein
submanifolds as a special case. In this direction, Hu--Li--Xing \cite{HLX22}
classified minimal Legendrian submanifolds in $\mathbb S^{2n+1}$ with parallel
Ricci tensor for $n=3,4$. From the viewpoint of submanifold geometry, it is also
natural to retain constant sectional curvature while relaxing minimality to a
geometric condition on the mean curvature vector. This leads to the following
problem.

\begin{problem}
Classify all Legendrian submanifolds in the unit sphere $\mathbb S^{2n+1}$
with conformal Maslov form and constant sectional curvature for $n\ge2$.
\end{problem}

Some facts relevant to the above problem are as follows. A Legendrian submanifold
with mean curvature vector $H$ is said to have \textit{conformal Maslov form} if
the tangent vector field $T:=-\varphi H$ is conformal, or equivalently, $\nabla_XT
=\rho X$ with $\rho=\tfrac1n\operatorname{div}T$ (cf. \cite{Cas96,CMU01,CU93,Pit05}).
This includes minimal submanifolds and those with $C$-parallel mean curvature vector,
for which $\nabla T=0$. See Section \ref{sect:2.2} for details. Cheng--Li--Hu
classified such submanifolds in dimension two \cite[Theorem 1.8]{CLH26} and
those with $C$-parallel mean curvature vector in dimensions $n\geq3$
\cite[Theorem 1.11]{CLH26}.

Motivated by these results, we solve the above problem in all dimensions
$n\geq2$. Before stating the first main theorem, we shall introduce another
Legendrian submanifold that is not totally geodesic.

\begin{example}\label{exa:1.2}
Let $\alpha_1,\ldots,\alpha_{n+1}$ be positive constants satisfying $\sum_{a=1}^{n+1}
\alpha_a^2=1$ and denote the $n$-dimensional linear space by
\begin{equation}\label{eqn:1.3}
V_\alpha=\Big\{x=(x_1,\ldots,x_{n+1})\in\mathbb R^{n+1}:
\sum_{a=1}^{n+1}\alpha_a^2x_a=0\Big\}.
\end{equation}
Then the weighted Clifford immersion $\Psi_\alpha(x):V_\alpha\to\mathbb S^{2n+1}
\subset\mathbb C^{n+1}$, given by
\begin{equation}\label{eqn:1.4}
\Psi_\alpha(x)=(\alpha_1e^{\mathrm{i}x_1},\ldots,\alpha_{n+1}e^{\mathrm{i}x_{n+1}}),
\end{equation}
is a Legendrian immersion with flat induced metric and $C$-parallel second
fundamental form. Moreover, it is minimal if and only if $\alpha_1=\cdots=\alpha_{n+1}
=1/\sqrt{n+1}$.
\end{example}

\begin{remark}\label{rem:1.1}
The Calabi product submanifolds in \cite[Examples 1.5 and 1.6]{CLH26} are
locally represented by $\Psi_\alpha$. The weights give a common description
of these submanifolds. When all weights are equal, $\Psi_\alpha$ reduces to
the Clifford torus in Example \ref{exa:1.1}.
\end{remark}

We can now state our first main theorem, which completely solves the above problem.

\begin{theorem}\label{thm:1.1}
Let $M^n$ $(n\geq2)$ be a Legendrian submanifold in the unit sphere $\mathbb
S^{2n+1}$ with conformal Maslov form and constant sectional curvature.
Then $M^n$ is locally congruent either to a totally geodesic Legendrian sphere or
to the immersion $\Psi_\alpha$ given in Example \ref{exa:1.2}.
\end{theorem}

\begin{remark}\label{rem:1.2}
Theorem \ref{thm:1.1} extends Theorem A from minimality to the conformal Maslov
class. In the unit sphere, it also extends \cite[Theorem 1.8]{CLH26} to all dimensions
and removes the $C$-parallel mean curvature assumption in \cite[Theorem 1.11]{CLH26}.
\end{remark}

Moreover, the first pinching theorem of Dillen--Vrancken \cite{DV90} can be stated
as follows.

\begin{theoremB}[\cite{DV90}]
Let $M^n$ $(n\geq2)$ be a closed minimal Legendrian submanifold in the unit sphere
$\mathbb S^{2n+1}$. If its sectional curvature satisfies $0\leq\sec_g\leq1$, then
either $\sec_g=1$ or $\sec_g=0$ identically.
\end{theoremB}

Our second main theorem replaces minimality by the conformal Maslov form condition
and, in addition, determines the global submanifolds in both curvature alternatives.
We first describe the following compact flat Legendrian submanifold in $\mathbb
S^{2n+1}$.

\begin{example}\label{exa:1.3}
Let $\alpha=(\alpha_1,\ldots,\alpha_{n+1})$ be a positive integer vector and
put $\beta:=\sum_{a=1}^{n+1}\alpha_a$. Denote the $n$-dimensional torus by
\begin{equation}\label{eqn:1.5}
T^n_\alpha
=\Big\{(y_1,\ldots,y_{n+1})\in\mathbb S^1(1)\times\cdots\times\mathbb S^1(1):
\prod_{a=1}^{n+1}y_a^{\alpha_a}=1\Big\}.
\end{equation}
Then the weighted Clifford immersion $F_\alpha(y):T^n_\alpha\to\mathbb S^{2n+1}
\subset\mathbb C^{n+1}$, given by
\begin{equation}\label{eqn:1.6}
F_\alpha(y_1,\ldots,y_{n+1})
=(\sqrt{\tfrac{\alpha_1}{\beta}}y_1,\ldots,
\sqrt{\tfrac{\alpha_{n+1}}{\beta}}y_{n+1}),
\end{equation}
is a Legendrian embedding with flat induced metric and $C$-parallel second
fundamental form, whose image $F_\alpha(T^n_\alpha)$ is a  closed embedded
weighted Clifford Legendrian torus. The embedding is minimal if and only if
$\alpha_1=\cdots=\alpha_{n+1}=1$.
\end{example}

These compact Legendrian embeddings occur in the following pinching theorem.

\begin{theorem}\label{thm:1.2}
Let $M^n$ $(n\geq2)$ be a closed Legendrian submanifold in the unit sphere
$\mathbb S^{2n+1}$ with conformal Maslov form. If its sectional curvature
satisfies $0\leq\sec_g\leq1$, then one of the following occurs:
\begin{enumerate}
\item[$(a)$]
$M^n$ is a totally geodesic Legendrian sphere with $\sec_g=1$;

\item[$(b)$]
$M^n$ is a closed embedded weighted Clifford Legendrian torus with $\sec_g=0$.
\end{enumerate}
\end{theorem}

\begin{remark}\label{rem:1.3}
Cheng--Hu \cite{CH23} classified closed minimal Legendrian submanifolds in
$\mathbb S^{2n+1}$ with $\sec_g\geq0$ for $n\geq4$.
Theorem \ref{thm:1.2} replaces minimality by conformal Maslov form for $n\geq2$,
but requires the additional bound $\sec_g\leq1$.
\end{remark}

The rest of the paper is organized as follows. In Section \ref{sect:2}, we
first review the Sasakian structure of the unit sphere and the local theory
of Legendrian submanifolds, including the properties of the tensor $K$. In
Section \ref{sect:3}, we shall compute the geometric invariants of Example
\ref{exa:1.2} and describe its compact quotients. Section \ref{sect:4} collects
the properties needed for the proofs of the main results. Finally, Section
\ref{sect:5} completes the proofs of Theorems \ref{thm:1.1} and \ref{thm:1.2}.

\section{Preliminaries}\label{sect:2}

In this section, we first gather some essential material on Sasakian structure
$(\varphi,\xi,\eta,g)$ of the unit sphere $\mathbb{S}^{2n+1}$, which can be
realized as a Sasakian space form of constant $\varphi$-sectional curvature
$1$. We then briefly review the local theory of Legendrian submanifolds in
$\mathbb{S}^{2n+1}$, particularly regarding some properties of the tensor $K$.
For more details, we refer to \cite{CLH26,HLX22,LXY26} and the monograph
\cite{Bla10}. Throughout the paper, all the manifolds and submanifolds are
assumed to be connected, and closed means compact without boundary. Submanifolds
are understood to be embedded. The immersions are used through their local
embedded restrictions.

\subsection{Sasakian structure $(\varphi,\xi,\eta,g)$ of the unit sphere $\mathbb{S}^{2n+1}$}\label{sect:2.1}~

As a real hypersurface of the complex Euclidean space $\mathbb{C}^{n+1}$
with canonical complex structure $J$, the $(2n+1)$-dimensional unit sphere
$\mathbb{S}^{2n+1}$ naturally admits a Sasakian structure $(\varphi,\xi,\eta,
g)$, where $\xi=J\bar N$ is the structure vector field with the inward unit
normal vector field $\bar N$ of the inclusion $\mathbb{S}^{2n+1}\hookrightarrow
\mathbb{C}^{n+1}$, and $g$ denotes the induced metric on $\mathbb{S}^{2n+1}$.
Moreover, $\eta(X)=g(X,\xi)$ and $\varphi X=JX-\langle JX,\bar N\rangle\bar N$
for any tangent vector field $X$ on $\mathbb{S}^{2n+1}$, where $\langle\cdot,
\cdot\rangle$ denotes the real part of the standard Hermitian metric on
$\mathbb{C}^{n+1}$. In particular, for any tangent vector fields $X,Y$
on $\mathbb{S}^{2n+1}$, the Sasakian structure $(\varphi,\xi,\eta,g)$
of $\mathbb{S}^{2n+1}$ satisfies the following properties:
\begin{equation}\label{eqn:2.1}
\begin{cases}
g(\varphi X,\varphi Y)=g(X,Y)-\eta(X)\eta(Y),\\
\varphi\xi=0,\ \ \eta(\varphi X)=0,\ \ {\rm rank}\,(\varphi)=2n,\\
\varphi^2X=-X+\eta(X)\xi,\ \ d\eta(X,Y)=2g(X,\varphi Y),\\
\bar{\nabla}_{X}\xi=-\varphi X,\ \
(\bar{\nabla}_{X}\varphi)Y=g(X,Y)\xi-\eta(Y)X,
\end{cases}
\end{equation}
where $\bar{\nabla}$ is the Levi-Civita connection with respect to the induced
metric $g$ on $\mathbb{S}^{2n+1}$.

\subsection{Local theory of Legendrian submanifolds in the unit sphere $\mathbb{S}^{2n+1}$}\label{sect:2.2}~

Let $M^n$ be a Legendrian submanifold in the unit sphere $\mathbb{S}^{2n+1}$,
i.e., the contact form $\eta$ restricted to $M^n$ vanishes. Consequently,
$\xi$ is a normal vector field along $M^n$. Denote by $N$ a normal vector
field over $M^n$, and by $U,X,Y,Z$ the tangent vector fields on $M^n$ in
the subsequent paragraphs. Then we have the Gauss and Weingarten formulas:
\begin{equation}\label{eqn:2.2}
\bar\nabla_XY=\nabla_XY+h(X,Y),\ \
\bar\nabla_XN=-A_NX+\nabla_X^\perp N,
\end{equation}
where $\nabla$ is the Levi-Civita connection of the induced metric on $M^n$,
still denoted by $g$, $h$ (resp. $A_N$) is the second fundamental form (resp.
the shape operator with respect to $N$) of $M^n\to\mathbb{S}^{2n+1}$, and
$\nabla^\perp$ is the normal connection in the normal bundle $T^\perp M^n$.
Then, by means of \eqref{eqn:2.2} there holds
\begin{equation}\label{eqn:2.3}
g(h(X,Y),N)=g(A_NX,Y).
\end{equation}
Note from the facts $\eta(X)=0$ and $d\eta(X,Y)=2g(X,\varphi Y)$ that $\varphi$
maps the tangent vector fields of $M^n$ to the normal vector fields in $T^\perp
M^n$. Together with \eqref{eqn:2.1}, it satisfies
\begin{equation}\label{eqn:2.4}
\nabla^\perp_X\varphi Y=\varphi\nabla_XY+g(X,Y)\xi,\ \
A_{\varphi X}Y=-\varphi h(X,Y)=A_{\varphi Y}X,
\end{equation}
and thus $g(h(X,Y),\varphi Z)$ is totally symmetric in $X$, $Y$ and $Z$:
\begin{equation}\label{eqn:2.5}
g(h(X,Y),\varphi Z)=g(h(X,Z),\varphi Y)=g(h(Y,Z),\varphi X).
\end{equation}
The combination of \eqref{eqn:2.1}, \eqref{eqn:2.3} and the Weingarten
formula further implies that
\begin{equation}\label{eqn:2.6}
g(h(X,Y),\xi)=g(A_\xi X,Y)=0.
\end{equation}
In particular, we have $T^\perp M^n=\varphi(TM^n)\oplus\mathbb R\xi$.

Moreover, the equations of Gauss, Ricci and Codazzi are respectively
given by
\begin{gather}
R(X,Y)Z=g(Y,Z)X-g(X,Z)Y+[A_{\varphi X},A_{\varphi Y}]Z,
\label{eqn:2.7}\\
R^\perp(X,Y)\varphi Z=\varphi[A_{\varphi X},A_{\varphi Y}]Z,
\label{eqn:2.8}\\
(\bar\nabla h)(X,Y,Z)=(\bar\nabla h)(Y,X,Z),
\label{eqn:2.9}
\end{gather}
where
\begin{gather}
[A_{\varphi X},A_{\varphi Y}]
=A_{\varphi X}A_{\varphi Y}
-A_{\varphi Y}A_{\varphi X},\label{eqn:2.10}\\
R(X,Y)Z=\nabla_X\nabla_YZ-\nabla_Y\nabla_XZ
-\nabla_{[X,Y]}Z,\label{eqn:2.11}\\
R^\perp(X,Y)\varphi Z
=\nabla^\perp_X\nabla^\perp_Y\varphi Z
-\nabla^\perp_Y\nabla^\perp_X\varphi Z
-\nabla^\perp_{[X,Y]}\varphi Z,\label{eqn:2.12}\\
(\bar\nabla h)(X,Y,Z)
=\nabla^\perp_Xh(Y,Z)
-h(\nabla_XY,Z)-h(Y,\nabla_XZ).\label{eqn:2.13}
\end{gather}
Contracting the Gauss equation \eqref{eqn:2.7} twice, we obtain
\begin{equation}\label{eqn:2.14}
n(n-1)\kappa=n(n-1)+n^2\|H\|^2-\|h\|^2,
\end{equation}
where $\kappa$ is the normalized scalar curvature, $H:=\tfrac1n\operatorname
{trace} h$ is the mean curvature vector, and $\|\cdot\|^2$ denotes the squared
norm with respect to the metric $g$.

Following the usual notation for Legendrian submanifolds, we can define the
covariant derivative $\bar\nabla^\xi$ by
\begin{equation}\label{eqn:2.15}
(\bar\nabla^\xi h)(X,Y,Z)
:=(\bar\nabla h)(X,Y,Z)-g(h(Y,Z),\varphi X)\xi.
\end{equation}
With this definition, we say that $h$ is \textit{$C$-parallel} if $\bar\nabla^\xi
h=0$ on $M^n$. Similarly, the mean curvature vector is called \textit{$C$-parallel}
if it satisfies
\begin{equation}\label{eqn:2.16}
\nabla^\perp_XH=g(H,\varphi X)\xi.
\end{equation}
Set the tangent vector $T:=-\varphi H$, so that $H:=\varphi T$. According to
\eqref{eqn:2.4}, it follows that
\begin{equation}\label{eqn:2.17}
\nabla^\perp_XH=\varphi\nabla_XT+g(T,X)\xi.
\end{equation}
Consequently, $H$ is $C$-parallel if and only if $\nabla T=0$.

Furthermore, the one form dual to $T$ is given by
\begin{equation}\label{eqn:2.18}
\alpha_H(X):=g(T,X)=g(H,\varphi X).
\end{equation}
Using \eqref{eqn:2.5} and the Codazzi equation \eqref{eqn:2.9}, the tensor
$g((\bar\nabla h)(X,Y,Z),\varphi U)$ is totally symmetric. Taking its trace
gives $g(\nabla_XT,Y)=g(\nabla_YT,X)$, and hence
\begin{equation}\label{eqn:2.19}
d\alpha_H(X,Y):=g(\nabla_XT,Y)-g(\nabla_YT,X)=0.
\end{equation}
Following Piti\c{s} \cite{Pit05} (see also \cite{CLH26}), $M^n$ is said to have
\textit{conformal Maslov form} if and only if $T$ is a conformal vector field.
Since $\alpha_H$ is closed, this is equivalent to
\begin{equation}\label{eqn:2.20}
\nabla_XT=\rho X,\ \
\rho=\tfrac1n\operatorname{div}T.
\end{equation}
In particular, this class contains minimal submanifolds and submanifolds with
$C$-parallel mean curvature vector, for which $\rho=0$.

\subsection{Properties of the tensor $K$ on Legendrian submanifolds}\label{sect:2.3}~

On a Legendrian submanifold $M^n$ in $\mathbb{S}^{2n+1}$, we define the tensor
$K:=-\varphi h$ as in \cite{LXY26} and thus $h(X,Y)=\varphi K(X,Y)$ for $X,Y\in
TM^n$. Then the following properties are immediate.

\begin{lemma}[\cite{LXY26}]\label{lem:2.1}
For the tensor $K$ of the Legendrian submanifold $M^n$ in $\mathbb{S}^{2n+1}$,
\begin{enumerate}
\item[(1)]
$K_XY=K(X,Y)=A_{\varphi X}Y$ and $g(K(X,Y),Z)$ is totally symmetric;

\item[(2)]
$M^n$ is minimal if and only if $\operatorname{trace}K_X=0$ for any $X\in TM^n$;

\item[(3)]
The Codazzi equation for the tensor $K$ is given by
\begin{equation}\label{eqn:2.21}
(\nabla K)(X,Y,Z)=(\nabla K)(Y,X,Z);
\end{equation}

\item[(4)]
The Ricci identity for the tensor $K$ is given by
\begin{equation}\label{eqn:2.22}
\begin{aligned}
&(\nabla^2K)(U,X,Y,Z)-(\nabla^2K)(X,U,Y,Z)=(R\cdot K)(U,X,Y,Z)\\
&=R(U,X)K(Y,Z)-K(R(U,X)Y,Z)-K(Y,R(U,X)Z),
\end{aligned}
\end{equation}
where $(\nabla^2K)(U,X,Y,Z)=(\nabla_U(\nabla K))(X,Y,Z)$ for any $U,X,Y,Z\in TM^n$.
\end{enumerate}
\end{lemma}

The covariant derivative of $K$ is defined by
\begin{equation}\label{eqn:2.23}
(\nabla K)(X,Y,Z)=\nabla_XK(Y,Z)
-K(\nabla_XY,Z)-K(Y,\nabla_XZ).
\end{equation}
In particular, the cubic form
\begin{equation}\label{eqn:2.24}
K^\flat(X,Y,Z):=g(K(X,Y),Z)
\end{equation}
and its covariant derivative are totally symmetric. Moreover,
\begin{equation}\label{eqn:2.25}
\operatorname{trace}K_X=ng(T,X),\ \
\operatorname{trace}K=nT,\ \  H=\varphi T.
\end{equation}
The Gauss equation \eqref{eqn:2.7} can be written as
\begin{equation}\label{eqn:2.26}
R(X,Y)Z=g(Y,Z)X-g(X,Z)Y+[K_X,K_Y]Z,
\end{equation}
where $[K_X,K_Y]=K_XK_Y-K_YK_X$. Since $\|K\|^2=\|h\|^2$ and $\|T\|
=\|H\|$, \eqref{eqn:2.14} gives
\begin{equation}\label{eqn:2.27}
n(n-1)\kappa=n(n-1)+n^2\|T\|^2-\|K\|^2.
\end{equation}
By means of \eqref{eqn:2.15}, it can be checked by direct
calculations that
\begin{equation}\label{eqn:2.28}
(\bar\nabla^\xi h)(X,Y,Z)=\varphi(\nabla K)(X,Y,Z).
\end{equation}
Consequently, $\nabla K=0$ if and only if $\bar\nabla^\xi h=0$.

We next introduce the traceless part of the tensor $K$. Define
\begin{equation}\label{eqn:2.29}
\begin{aligned}
\tilde{K}(X,Y):=K(X,Y)-\tfrac{n}{n+2}\{g(T,X)Y+g(T,Y)X+g(X,Y)T\}.
\end{aligned}
\end{equation}
By construction,
\begin{equation}\label{eqn:2.30}
\operatorname{trace}\tilde{K}_X=0,\ \ X\in TM^n.
\end{equation}
If $M^n$ has conformal Maslov form, differentiating \eqref{eqn:2.29} yields
\begin{equation}\label{eqn:2.31}
\begin{aligned}
g((\nabla_U\tilde K)(X,Y),Z)
=~&g((\nabla_UK)(X,Y),Z)-\tfrac{n\rho}{n+2}\{g(U,X)g(Y,Z)\\
~&+g(U,Y)g(X,Z)+g(U,Z)g(X,Y)\}.
\end{aligned}
\end{equation}
Thus $\tilde K^\flat$ and $\nabla\tilde K^\flat$ are totally symmetric and
traceless.

Finally, let $Q$ be a smooth tensor field of any fixed type on $M^n$. Its
second covariant derivative is defined by
\begin{equation}\label{eqn:2.32}
(\nabla^2Q)(X,Y):=\nabla_X(\nabla_YQ)-\nabla_{\nabla_XY}Q.
\end{equation}
If $\{E_i\}_{i=1}^n$ is a local orthonormal frame, the Laplacian of $Q$ is
the metric trace
\begin{equation}\label{eqn:2.33}
\Delta Q=\operatorname{trace}\nabla^2Q
=\sum_{i=1}^n(\nabla^2Q)(E_i,E_i).
\end{equation}
With the trace convention in \eqref{eqn:2.33}, the standard Bochner identity
reads (cf. \cite{Bes87})
\begin{equation}\label{eqn:2.34}
\tfrac12\Delta\|Q\|^2=\|\nabla Q\|^2+g(\Delta Q,Q).
\end{equation}

\section{Weighted Clifford Legendrian immersions}\label{sect:3}

In this section, we compute the geometric invariants of the weighted Clifford
Legendrian immersion in Example \ref{exa:1.2} and explain its relation to that
in Example \ref{exa:1.3}.

\begin{proposition}\label{pro:3.1}
For $n\geq2$, the immersion $\Psi_\alpha$ in Example \ref{exa:1.2} defines a
Legendrian submanifold in the unit sphere $\mathbb S^{2n+1}$ with flat induced
metric and $C$-parallel second fundamental form. In particular, it is minimal
if and only if $\alpha_1=\cdots=\alpha_{n+1}=1/\sqrt{n+1}$.
\end{proposition}

\begin{proof}
We introduce local coordinates $(u_1,\ldots,u_n)$ on $V_\alpha$ by
\begin{equation}\label{eqn:3.1}
x_j=u_j,\ \ 1\leq j\leq n,\ \
x_{n+1}=-\tfrac{1}{\alpha_{n+1}^2}\sum_{j=1}^n\alpha_j^2u_j.
\end{equation}
Let $\{c_a\}_{a=1}^{n+1}$ be the standard unit vectors of $\mathbb C^{n+1}$.
Direct differentiation yields
\begin{equation}\label{eqn:3.2}
\Psi_{\alpha,u_j}=\tfrac{\partial\Psi_\alpha}{\partial u_j}
=\mathrm{i}\alpha_je^{\mathrm{i}u_j}c_j
-\mathrm{i}\tfrac{\alpha_j^2}{\alpha_{n+1}}
e^{\mathrm{i}x_{n+1}}c_{n+1},\ \ 1\leq j\leq n.
\end{equation}
Regard $\alpha$ as a unit vector in $\mathbb R^{n+1}$. Choose an
orthonormal basis $\{q_j\}_{j=1}^n$ of $\alpha^\perp$ and write
$q_j=(q_{1j},\ldots,q_{n+1,j})$. Then
$\{q_1,\ldots,q_n,\alpha\}$ is an orthonormal basis of
$\mathbb R^{n+1}$. Define linear coordinates $(w_1,\ldots,w_n)$ on
$V_\alpha$ by
\begin{equation}\label{eqn:3.3}
x_a=\tfrac{1}{\alpha_a}\sum_{j=1}^nq_{aj}w_j,
\ \ 1\leq a\leq n+1.
\end{equation}
Indeed, $\sum_{a=1}^{n+1}\alpha_a^2x_a=0$ follows from $\sum_{a=1}^{n+1}
\alpha_aq_{aj}=0$. Applying \eqref{eqn:3.2} we deduce that
\begin{equation}\label{eqn:3.4}
\begin{aligned}
E_j:=\Psi_{\alpha,w_j}
=\sum_{k=1}^n\tfrac{q_{kj}}{\alpha_k}\Psi_{\alpha,u_k}
=\mathrm{i}\sum_{a=1}^{n+1}q_{aj}e^{\mathrm{i}x_a}c_a,
\ \ 1\leq j\leq n,
\end{aligned}
\end{equation}
form an orthonormal frame. More precisely, direct calculations show that
\begin{equation}\label{eqn:3.5}
\langle E_j,E_k\rangle=\delta_{jk},\ \
\langle E_j,\Psi_\alpha\rangle
=\langle E_j,\mathrm{i}\Psi_\alpha\rangle
=\langle E_j,\mathrm{i}E_k\rangle=0,
\ \ 1\leq j,k\leq n,
\end{equation}
where $\langle\cdot,\cdot\rangle$ is the standard Euclidean inner product
on $\mathbb C^{n+1}\cong\mathbb R^{2n+2}$. Consequently, $\Psi_\alpha$
is a Legendrian immersion into $\mathbb S^{2n+1}$ and
$\varphi E_j=\mathrm{i}E_j$.

Moreover, writing
$\Psi_{\alpha,w_jw_k}=\partial^2\Psi_\alpha/\partial w_j\partial w_k$ for
$1\leq j,k\leq n$, by differentiation and the Gauss formula, we obtain
\begin{equation}\label{eqn:3.6}
\begin{aligned}
\Psi_{\alpha,w_jw_k}
=-\sum_{a=1}^{n+1}\tfrac{q_{aj}q_{ak}}{\alpha_a}
e^{\mathrm{i}x_a}c_a
=\nabla_{E_j}E_k+h(E_j,E_k)-\delta_{jk}\Psi_\alpha
\end{aligned}
\end{equation}
for $1\leq j,k\leq n$. It then follows that
\begin{equation}\label{eqn:3.7}
\nabla_{E_j}E_k=0,\ \ 1\leq j,k\leq n,
\end{equation}
which implies that $\Psi_\alpha$ has flat induced metric, and for its second
fundamental form $h$,
\begin{equation}\label{eqn:3.8}
\begin{aligned}
h(E_j,E_k)
=\sum_{\ell=1}^nh^{\ell*}_{jk}\varphi E_\ell,\ \
h^{\ell*}_{jk}:=g(h(E_j,E_k),\varphi E_\ell)
=\sum_{a=1}^{n+1}\tfrac{q_{aj}q_{ak}q_{a\ell}}{\alpha_a}.
\end{aligned}
\end{equation}
As $\{q_1,\ldots,q_n,\alpha\}$ is an orthonormal basis of $\mathbb R^{n+1}$,
$\sum_{j=1}^nq_{aj}^2=1-\alpha_a^2$. Consequently, the mean curvature vector
$H$ of $\Psi_\alpha$ satisfies
\begin{equation}\label{eqn:3.9}
\begin{aligned}
nH=\sum_{j=1}^nh(E_j,E_j)
=\sum_{\ell=1}^n
\bigl(\sum_{a=1}^{n+1}\tfrac{q_{a\ell}}{\alpha_a}\bigr)
\varphi E_\ell.
\end{aligned}
\end{equation}
Together with \eqref{eqn:3.7} and \eqref{eqn:3.8}, we see from \eqref{eqn:2.4}
that $\nabla^\perp_{E_i}\varphi E_\ell=\delta_{i\ell}\xi$ and hence
\begin{equation}\label{eqn:3.10}
\begin{aligned}
(\bar\nabla h)(E_i,E_j,E_k)
=g(h(E_j,E_k),\varphi E_i)\xi,\ \ 1\leq i,j,k\leq n.
\end{aligned}
\end{equation}
Thus $\bar\nabla^\xi h=0$ because of \eqref{eqn:2.15}, that is, $\Psi_\alpha$
has $C$-parallel second fundamental form.

Finally, put $v:=(\alpha_1^{-1},\ldots,\alpha_{n+1}^{-1})$. Based on \eqref{eqn:3.9},
$H=0$ if and only if $v$ is orthogonal to every $q_\ell$, and hence if and only
if $v=c\alpha$ for some positive constant $c$. Thus $c\alpha_a^2=1$ for $1\leq
a\leq n+1$. Summing these identities gives $c=n+1$, and hence $\alpha_1=\cdots
=\alpha_{n+1}=1/\sqrt{n+1}$. The converse follows immediately from \eqref{eqn:3.9}.
For these equal weights, \eqref{eqn:3.1} reduces to $x_{n+1}=-(u_1+\cdots+u_n)$,
and $\Psi_\alpha$ becomes the immersion $F_1$ in Example \ref{exa:1.1}. This
completes the proof of Proposition \ref{pro:3.1}.
\end{proof}

\begin{remark}\label{rem:3.1}
To distinguish the continuous weights in Example \ref{exa:1.2} from the integer
parameter in Example \ref{exa:1.3}, denote the latter here by $\gamma=(\gamma_1,
\ldots,\gamma_{n+1})$. The immersion $\Psi_\alpha$ descends to a closed quotient
if and only if its period group $P_\alpha=2\pi\mathbb Z^{n+1}\cap V_\alpha$ has
rank $n$. This is equivalent to the existence of a primitive positive integer
vector $\gamma$ such that $\alpha_a^2=\gamma_a/\sum_{b=1}^{n+1}\gamma_b$ for
$1\leq a\leq n+1$. Exponentiation identifies the canonical quotient $V_\alpha
/P_\alpha$ with $T^n_\gamma$, and the descended map is precisely the embedding
$F_\gamma$ in Example \ref{exa:1.3}. For the descended map to be an embedding,
the quotient must be taken by the full period group $P_\alpha$. Therefore, the
closed embedded submanifolds are exactly those in Example \ref{exa:1.3}.
\end{remark}

\section{Legendrian submanifolds with conformal Maslov form}\label{sect:4}

In this section, we establish the properties needed for the proofs of the main
theorems. We treat constant sectional curvature, distinguishing the positive, flat
and negative cases, and then obtain the integral and curvature identities for the
closed pinching problem.

\subsection{Conformal Maslov form and constant sectional curvature}\label{sect:4.1}~

We obtain the differential consequences derived from the conformal Maslov equation.

\begin{lemma}\label{lem:4.1}
Let $M^n$ $(n\geq2)$ be a Legendrian submanifold in the unit sphere $\mathbb S^{2n+1}$
with conformal Maslov form and constant sectional curvature $c$. Then
\begin{equation}\label{eqn:4.1}
\nabla \rho=-cT,\ \
\operatorname{Hess}\rho=-c\rho g,\ \
\Delta T=-cT.
\end{equation}
\end{lemma}

\begin{proof}
Commuting covariant derivatives in the conformal Maslov equation \eqref{eqn:2.20}
and using the constant curvature formula for $R(X,Y)T$, we deduce
that
\begin{equation}\label{eqn:4.2}
X(\rho)Y-Y(\rho)X=c\{g(Y,T)X-g(X,T)Y\},
\end{equation}
or equivalently,
\begin{equation}\label{eqn:4.3}
(d\rho+cT^\flat)(X)Y=(d\rho+cT^\flat)(Y)X,\ \ X,Y\in TM^n.
\end{equation}
This implies that the one form $(d\rho+cT^\flat)$ vanishes for $n\ge2$, and
therefore $\nabla \rho=-cT$. By differentiating this identity and using
\eqref{eqn:2.20}, we further have
\begin{equation}\label{eqn:4.4}
\operatorname{Hess}\rho(X,Y)
=\nabla_X(\nabla_Y\rho)
-(\nabla_XY)(\rho)
=-c\rho g(X,Y),
\end{equation}
which proves the second identity in \eqref{eqn:4.1}. Moreover, differentiation
of \eqref{eqn:2.20} results in
\begin{equation}\label{eqn:4.5}
(\nabla^2T)(X,Y)
=\nabla_X(\nabla_YT)
-\nabla_{\nabla_XY}T
=X(\rho)Y.
\end{equation}
Taking the metric trace and using the first identity in \eqref{eqn:4.1}, we
obtain $\Delta T=\nabla \rho=-cT$. This proves the third identity and completes
the proof of Lemma \ref{lem:4.1}.
\end{proof}

Recall the traceless tensor $\tilde K$ from \eqref{eqn:2.29}. The following
proposition collects the basic Laplacian and norm identities for $K$ and
$\tilde K$.

\begin{proposition}\label{pro:4.1}
Let $M^n$ $(n\geq2)$ be a Legendrian submanifold in the unit sphere $\mathbb
S^{2n+1}$ with conformal Maslov form and constant sectional curvature $c$. Then
\begin{equation}\label{eqn:4.6}
\begin{aligned}
(\Delta K)(X,Y)=(n+1)c K(X,Y)-nc\{
g(T,X)Y+g(T,Y)X+g(X,Y)T\}.
\end{aligned}
\end{equation}
In particular, we have
\begin{equation}\label{eqn:4.7}
\Delta\tilde{K}=(n+1)c\tilde{K}.
\end{equation}
Moreover, it holds that
\begin{equation}\label{eqn:4.8}
\|\tilde{K}\|^2=\tfrac{n^2(n-1)}{n+2}\|T\|^2
+n(n-1)(1-c).
\end{equation}
Furthermore, the invariant
\begin{equation}\label{eqn:4.9}
\rho^2+c\|T\|^2
\end{equation}
is constant.
\end{proposition}

\begin{proof}
Define the divergence of $K$ by
\begin{equation}\label{eqn:4.10}
(\operatorname{div}K)(X)
=\sum^n_{i=1}(\nabla_{E_i}K)(E_i,X)
=n\nabla_XT,
\end{equation}
where the last equality follows from the full symmetry of $g((\nabla_UK)(X,Y),
Z)$. Lower the upper index of $K$ and write
\begin{equation}\label{eqn:4.11}
K^\flat(X,Y,Z):=g(K(X,Y),Z).
\end{equation}
At a fixed point, we choose a local normal $g$-orthonormal frame $\{E_i\}_{i=1}^n$
and the full symmetry of $\nabla K^\flat$ gives
\begin{equation}\label{eqn:4.12}
(\Delta K)(X,Y)
=\sum_{i=1}^n(\nabla^2K)(E_i,E_i,X,Y)
=\sum_{i=1}^n(\nabla^2K)(E_i,X,E_i,Y).
\end{equation}
Applying the Ricci identity \eqref{eqn:2.22} we obtain
\begin{equation}\label{eqn:4.13}
\begin{aligned}
g((\Delta K)(X,Y),Z)
=~&\sum_{i=1}^ng((\nabla^2K)(X,E_i,E_i,Y),Z)
-\sum_{i=1}^nK^\flat(R(E_i,X)E_i,Y,Z)\\
&-\sum_{i=1}^nK^\flat(E_i,R(E_i,X)Y,Z)
-\sum_{i=1}^nK^\flat(E_i,Y,R(E_i,X)Z).
\end{aligned}
\end{equation}
Differentiating \eqref{eqn:4.10} shows that the first term on the right side
is
\begin{equation}\label{eqn:4.14}
\sum_{i=1}^ng((\nabla^2K)(X,E_i,E_i,Y),Z)
=n g((\nabla^2T)(X,Y),Z).
\end{equation}
Since $(M^n,g)$ has constant sectional curvature $c$,
\begin{equation}\label{eqn:4.15}
\begin{aligned}
-\sum_{i=1}^nK^\flat(R(E_i,X)E_i,Y,Z)
=(n-1)cK^\flat(X,Y,Z).
\end{aligned}
\end{equation}
The contractions in the second and third arguments are, respectively,
\begin{equation}\label{eqn:4.16}
\begin{aligned}
&-\sum_{i=1}^nK^\flat(E_i,R(E_i,X)Y,Z)
=cK^\flat(X,Y,Z)-nc\,g(X,Y)T^\flat(Z),\\
&-\sum_{i=1}^nK^\flat(E_i,Y,R(E_i,X)Z)
=cK^\flat(X,Y,Z)-nc\,g(X,Z)T^\flat(Y).
\end{aligned}
\end{equation}
Here we used the fact $\sum_{i=1}^nK^\flat(E_i,E_i,\cdot)=nT^\flat$ and the
total symmetry of $K^\flat$. Therefore, by combining the preceding identities
we have
\begin{equation}\label{eqn:4.17}
\begin{aligned}
(\Delta K)(X,Y)
=n(\nabla^2T)(X,Y)+(n+1)cK(X,Y)
-nc\{g(T,Y)X+g(X,Y)T\}.
\end{aligned}
\end{equation}
Moreover, we deduce from \eqref{eqn:4.1} and \eqref{eqn:4.5} that
\begin{equation}\label{eqn:4.18}
(\nabla^2T)(X,Y)
=g(\nabla \rho,X)Y
=-cg(T,X)Y.
\end{equation}
Substituting \eqref{eqn:4.18} into \eqref{eqn:4.17} immediately proves
\eqref{eqn:4.6}.

Define the tensor $\Phi$ by
\begin{equation}\label{eqn:4.19}
\Phi(X,Y):=g(T,X)Y+g(T,Y)X+g(X,Y)T.
\end{equation}
Since $\nabla g=0$, its first covariant derivative satisfies
\begin{equation}\label{eqn:4.20}
\begin{aligned}
(\nabla_U\Phi)(X,Y)
=g(\nabla_UT,X)Y+g(\nabla_UT,Y)X+g(X,Y)\nabla_UT.
\end{aligned}
\end{equation}
It follows that
\begin{equation}\label{eqn:4.21}
\begin{aligned}
(\Delta\Phi)(X,Y)
&=g(\Delta T,X)Y+g(\Delta T,Y)X+g(X,Y)\Delta T\\
&=-c\{g(T,X)Y+g(T,Y)X+g(X,Y)T\}=-c\Phi(X,Y),
\end{aligned}
\end{equation}
where the second line follows from $\Delta T=-cT$. We then check that
\begin{equation}\label{eqn:4.22}
\begin{aligned}
\Delta\tilde K
&=\Delta K-\tfrac{n}{n+2}\Delta\Phi\\
&=(n+1)cK-nc\Phi+\tfrac{nc}{n+2}\Phi\\
&=(n+1)c(K-\tfrac{n}{n+2}\Phi)
=(n+1)c\tilde K.
\end{aligned}
\end{equation}
This verifies the assertion in \eqref{eqn:4.7}.

We next calculate the norm of the traceless part. Fix a point and an
orthonormal basis $\{e_i\}_{i=1}^n$. By the definition of $\Phi$,
\begin{equation}\label{eqn:4.23}
\Phi(e_i,e_j)=g(T,e_i)e_j+g(T,e_j)e_i+g(e_i,e_j)T.
\end{equation}
The trace identity $\sum^n_{j=1}K(e_j,e_j)=nT$ and the total symmetry
show that
\begin{equation}\label{eqn:4.24}
\begin{aligned}
g(K,\Phi)
&=\sum^n_{i,j=1}g(K(e_i,e_j),\Phi(e_i,e_j))\\
&=2\sum^n_{i=1}g(T,e_i)\sum^n_{j=1}g(K(e_i,e_j),e_j)
+\sum^n_{i=1}g(K(e_i,e_i),T)=3n\|T\|^2.
\end{aligned}
\end{equation}
For the norm of $\Phi$, the squares of its three summands contribute
$3n\|T\|^2$, whereas their three pairwise cross terms contribute
$6\|T\|^2$. Consequently,
\begin{equation}\label{eqn:4.25}
\|\Phi\|^2=3(n+2)\|T\|^2.
\end{equation}
Combining these identities, by definition we have
\begin{equation}\label{eqn:4.26}
\begin{aligned}
\|\tilde K\|^2
&=\|K\|^2-\tfrac{2n}{n+2}g(K,\Phi)
+\tfrac{n^2}{(n+2)^2}\|\Phi\|^2\\
&=\|K\|^2-\tfrac{3n^2}{n+2}\|T\|^2.
\end{aligned}
\end{equation}
As the normalized scalar curvature is $c$, \eqref{eqn:2.27} reduces to
$\|K\|^2=n^2\|T\|^2+n(n-1)(1-c)$. Substitution into the last line of
\eqref{eqn:4.26} proves \eqref{eqn:4.8}.

It remains to establish the asserted constancy. It is seen from \eqref{eqn:2.20}
and \eqref{eqn:4.1} that
\begin{equation}\label{eqn:4.27}
d\rho=-c T^\flat,\ \ d\|T\|^2=2\rho T^\flat.
\end{equation}
Hence the quantity in \eqref{eqn:4.9} is constant, which completes the
proof of Proposition \ref{pro:4.1}.
\end{proof}

\subsection{Submanifolds with constant sectional curvature $c>0$}\label{sect:4.2}~

For $c>0$, the argument rests on an identity of Simons' type for the traceless
cubic form associated with $\tilde K$. As stated above, we write
\begin{equation}\label{eqn:4.28}
\tilde K^\flat(X,Y,Z):=g(\tilde K(X,Y),Z).
\end{equation}

\begin{lemma}\label{lem:4.2}
Let $M^n$ $(n\geq2)$ be a Legendrian submanifold in the unit sphere $\mathbb
S^{2n+1}$ with conformal Maslov form and constant sectional curvature $c$.
Put $S:=\nabla\tilde K^\flat$. Then
\begin{equation}\label{eqn:4.29}
0=\|\nabla S\|^2+4(n+2)c\|S\|^2
+2(n+1)^2c^2\|\tilde{K}\|^2.
\end{equation}
\end{lemma}

\begin{proof}
We first derive the Simons formula for $S$. Combining \eqref{eqn:2.20} and
\eqref{eqn:2.21}, we conclude from \eqref{eqn:2.29} and \eqref{eqn:2.30} that
$S$ is totally symmetric and traceless. Fix $p\in M^n$ and choose a local
orthonormal frame $\{E_i\}_{i=1}^n$ satisfying $\nabla E_i=0$ at $p$. Extend
$U,X,Y,Z$ locally so that their covariant derivatives vanish at $p$. Applying
the Ricci identity twice and noting the fact $\nabla R=0$, we obtain
\begin{equation}\label{eqn:4.30}
\begin{aligned}
(\Delta S)(U,X,Y,Z)
=~&g((\nabla_U\Delta\tilde K)(X,Y),Z)
+\sum_{i=1}^n(R(E_i,U)\cdot S)(E_i,X,Y,Z)\\
&+\sum_{i=1}^n\{R(E_i,U)\cdot
(\nabla_{E_i}\tilde K^\flat)\}(X,Y,Z).
\end{aligned}
\end{equation}
Here, for a covariant $k$-tensor $A$, we use the convention
\begin{equation}\label{eqn:4.31}
\begin{aligned}
(R(U,V)\mathbin{\cdot}A)(X_1,\ldots,X_k):=-\sum_{s=1}^k
A(X_1,\ldots,R(U,V)X_s,\ldots,X_k).
\end{aligned}
\end{equation}
Let $\mathop{\mathfrak{S}}_{X,Y,Z}$ be the cyclic summation over $X,Y,Z$.
Since $(M^n,g)$ has constant sectional curvature $c$, it follows that
\begin{equation}\label{eqn:4.32}
\begin{aligned}
\sum_{i=1}^n(R(E_i,U)\cdot S)(E_i,X,Y,Z)
=~&(n-1)cS(U,X,Y,Z)+c\mathop{\mathfrak{S}}_{X,Y,Z}S(X,U,Y,Z)\\
=~&(n+2)cS(U,X,Y,Z).
\end{aligned}
\end{equation}
Similarly,
\begin{equation}\label{eqn:4.33}
\begin{aligned}
\sum_{i=1}^n(R(E_i,U)\cdot
(\nabla_{E_i}\tilde K^\flat))(X,Y,Z)
=c\mathop{\mathfrak{S}}_{X,Y,Z}S(X,U,Y,Z)
=3cS(U,X,Y,Z).
\end{aligned}
\end{equation}
Moreover, \eqref{eqn:4.7} implies that
\begin{equation}\label{eqn:4.34}
\begin{aligned}
g((\nabla_U\Delta\tilde K)(X,Y),Z)=(n+1)cS(U,X,Y,Z).
\end{aligned}
\end{equation}
Substituting \eqref{eqn:4.32}--\eqref{eqn:4.34} into \eqref{eqn:4.30},
we immediately have
\begin{equation}\label{eqn:4.35}
\begin{aligned}
(\Delta S)(U,X,Y,Z)=2(n+3)cS(U,X,Y,Z).
\end{aligned}
\end{equation}

We next construct the scalar quantity needed for the final Bochner argument.
Through combining the Bochner identity \eqref{eqn:2.34} with \eqref{eqn:4.7},
we find
\begin{equation}\label{eqn:4.36}
\tfrac12\Delta\|\tilde{K}\|^2
 =\|S\|^2+(n+1)c\|\tilde{K}\|^2.
\end{equation}
On the other hand, using $\|\nabla T\|^2=n\rho^2$ and $\Delta T=-cT$,
we get
\begin{equation}\label{eqn:4.37}
\tfrac12\Delta\|T\|^2
=\|\nabla T\|^2+g(\Delta T,T)
=n\rho^2-c\|T\|^2,
\end{equation}
which, combined with \eqref{eqn:4.8}, gives
\begin{equation}\label{eqn:4.38}
\tfrac12\Delta\|\tilde K\|^2
=\tfrac{n^2(n-1)}{n+2}(n\rho^2-c\|T\|^2).
\end{equation}
Comparing \eqref{eqn:4.36} and \eqref{eqn:4.38} and using \eqref{eqn:4.8},
we obtain
\begin{equation}\label{eqn:4.39}
\begin{aligned}
\|S\|^2
=(n-1)\big\{\tfrac{n^3}{n+2}\rho^2-n^2c\|T\|^2
+cn(n+1)(c-1)\big\}.
\end{aligned}
\end{equation}
Adding $2(n+1)c\|\tilde K\|^2$ to \eqref{eqn:4.39} and simplifying with
\eqref{eqn:4.8}, we conclude that
\begin{equation}\label{eqn:4.40}
\|S\|^2+2(n+1)c\|\tilde{K}\|^2
=\tfrac{n(n-1)}{n+2}\{n^2(\rho^2+c\|T\|^2)
-(n+1)(n+2)c(c-1)\}.
\end{equation}

It remains to calculate the Laplacian of \eqref{eqn:4.40}, whose right side
is constant by \eqref{eqn:4.9}. Therefore, according to \eqref{eqn:2.34},
\eqref{eqn:4.7} and \eqref{eqn:4.35}, it is easy to check that
\begin{equation}\label{eqn:4.41}
\begin{aligned}
0=~&\tfrac12\Delta\{\|S\|^2+2(n+1)c\|\tilde K\|^2\}\\
=~&\|\nabla S\|^2+2(n+3)c\|S\|^2
+2(n+1)c\{\|S\|^2+(n+1)c\|\tilde K\|^2\}\\
=~&\|\nabla S\|^2+4(n+2)c\|S\|^2
+2(n+1)^2c^2\|\tilde K\|^2.
\end{aligned}
\end{equation}
This verifies \eqref{eqn:4.29} and hence completes the proof of Lemma
\ref{lem:4.2}.
\end{proof}

\subsection{Submanifolds with constant sectional curvature $c=0$}\label{sect:4.3}~

For $c=0$, lower the upper index of $K$ and define
\begin{equation}\label{eqn:4.42}
\begin{aligned}
K^\flat(X,Y,Z)&:=g(K(X,Y),Z).
\end{aligned}
\end{equation}
Set the $(0,4)$-tensor $P:=\nabla K^\flat$. It then follows that
\begin{equation}\label{eqn:4.43}
\begin{aligned}
P(U,X,Y,Z)
=(\nabla_UK^\flat)(X,Y,Z)
=g((\nabla_UK)(X,Y),Z),
\end{aligned}
\end{equation}
where the last equality follows from the fact $\nabla g=0$. Based on
the symmetry of $K^\flat$ and the Codazzi equation \eqref{eqn:2.21},
$P$ is totally symmetric. For $p\in M^n$ and $U,X\in T_pM^n$, let
$P_{UX}\in\operatorname{End}(T_pM^n)$ be determined by
\begin{equation}\label{eqn:4.44}
\begin{aligned}
g(P_{UX}Y,Z):=P(U,X,Y,Z).
\end{aligned}
\end{equation}
Then $P_{UX}=(\nabla_UK)_X$, while the total symmetry implies that $P_{UX}
=P_{XU}$ and makes $P_{UX}$ self-adjoint with respect to the metric $g$.
The following proposition verifies that $P$ is parallel and establishes the
commutator identity needed in the flat case.

\begin{proposition}\label{pro:4.2}
Let $M^n$ $(n\geq2)$ be a flat Legendrian submanifold in the unit sphere
$\mathbb S^{2n+1}$ with conformal Maslov form. Then
\begin{equation}\label{eqn:4.45}
\Delta P=0,\ \ \nabla P=0.
\end{equation}
Moreover, it holds that
\begin{equation}\label{eqn:4.46}
[P_{UX},P_{VY}]+[P_{VX},P_{UY}]=0.
\end{equation}
\end{proposition}

\begin{proof}
By Lemma \ref{lem:4.1}, $\rho$ is constant and $\Delta T=0$. For $c=0$,
\eqref{eqn:4.6} reduces to
\begin{equation}\label{eqn:4.47}
\Delta K=0.
\end{equation}
For $c=0$, \eqref{eqn:2.27} becomes $\|K\|^2=n^2\|T\|^2+n(n-1)$.
Applying \eqref{eqn:2.34} to $K$ and $T$, and using \eqref{eqn:4.47}
and $\Delta T=0$, we obtain
\begin{equation}\label{eqn:4.48}
\|P\|^2=\tfrac12\Delta\|K\|^2
=\tfrac{n^2}{2}\Delta\|T\|^2
=n^2\|\nabla T\|^2=n^3\rho^2.
\end{equation}
Since $\rho$ is constant, \eqref{eqn:4.48} shows that $P$ has constant
length. Flatness permits covariant differentiation to commute with the
connection Laplacian. Consequently, \eqref{eqn:4.47} yields
\begin{equation}\label{eqn:4.49}
\Delta P=\Delta(\nabla K^\flat)
=\nabla(\Delta K^\flat)=0.
\end{equation}
Applying \eqref{eqn:2.34} to $P$ and using \eqref{eqn:4.48} and
\eqref{eqn:4.49}, we further obtain
\begin{equation}\label{eqn:4.50}
0=\tfrac12\Delta\|P\|^2
=\|\nabla P\|^2+g(\Delta P,P)
=\|\nabla P\|^2.
\end{equation}
Consequently, $\nabla P=0$ and \eqref{eqn:4.45} follows immediately.

It remains to prove the commutator identity in \eqref{eqn:4.46}. Since $c=0$,
the Gauss equation \eqref{eqn:2.26} takes the tensor form
\begin{equation}\label{eqn:4.51}
[K_X,K_Y]=-X\wedge Y,
\end{equation}
where $(X\wedge Y)Z=g(Y,Z)X-g(X,Z)Y$. Fix $p\in M^n$ and extend $X,Y$ locally
so that $\nabla X=\nabla Y=0$ at $p$. For any $Z\in T_pM^n$,
\begin{equation}\label{eqn:4.52}
(\nabla_UK_X)Z
=(\nabla_UK)(X,Z)+K(\nabla_UX,Z)
=P_{UX}Z,
\end{equation}
and thus $\nabla_UK_X=P_{UX}$. Similarly, $\nabla_UK_Y=P_{UY}$.
Differentiating the Gauss equation in the $U$-direction, we calculate that
\begin{equation}\label{eqn:4.53}
\begin{aligned}
0&=\nabla_U[K_X,K_Y]\\
&=P_{UX}K_Y+K_XP_{UY}-P_{UY}K_X-K_YP_{UX}\\
&=[P_{UX},K_Y]+[K_X,P_{UY}].
\end{aligned}
\end{equation}
Since $p$ is arbitrary, \eqref{eqn:4.53} holds on $M^n$.

Fix $p\in M^n$ and extend $U,X,Y$ locally so that their covariant derivatives
vanish at $p$. Since $\nabla g=0$, covariant differentiation commutes with
raising the last index of $P$. Thus, for $A=X,Y$, the induced connection on
$\operatorname{End}(TM^n)$ satisfies at $p$
\begin{equation}\label{eqn:4.54}
\begin{aligned}
\nabla_V(P_{UA})
&=(\nabla_VP)_{UA}
+P_{\nabla_VU,A}+P_{U,\nabla_VA}
=(\nabla_VP)_{UA}=0,\\
\nabla_VK_A
&=(\nabla_VK)_A+K_{\nabla_VA}=P_{VA},
\end{aligned}
\end{equation}
where the last equality in the first identity used the fact $\nabla P=0$.
Differentiating \eqref{eqn:4.53} in the $V$-direction, we deduce from
\eqref{eqn:4.54} that
\begin{equation}\label{eqn:4.55}
\begin{aligned}
0&=\nabla_V\{[P_{UX},K_Y]+[K_X,P_{UY}]\}\\
&=P_{UX}P_{VY}-P_{VY}P_{UX}+P_{VX}P_{UY}-P_{UY}P_{VX}\\
&=[P_{UX},P_{VY}]+[P_{VX},P_{UY}].
\end{aligned}
\end{equation}
Since $p$ is arbitrary, this proves \eqref{eqn:4.46} and completes the
proof of Proposition \ref{pro:4.2}.
\end{proof}

The preceding proposition also determines a pointwise normal form for $P$.

\begin{lemma}\label{lem:4.3}
Let $M^n$ $(n\geq2)$ be a flat Legendrian submanifold in the unit sphere
$\mathbb S^{2n+1}$ with conformal Maslov form. For an arbitrary $p\in M^n$,
there is a $g$-orthonormal basis $\{e_i\}^n_{i=1}$ of $T_pM^n$, with the
dual basis $\{\omega^i\}$, such that, at $p$,
\begin{equation}\label{eqn:4.56}
P=n\rho\sum_{i=1}^n(\omega^i)^4.
\end{equation}
\end{lemma}

\begin{proof}
Taking $U=Y$ and $V=X$ in \eqref{eqn:4.46}, we obtain
\begin{equation}\label{eqn:4.57}
[P_{XX},P_{YY}]=0.
\end{equation}
Thus the self-adjoint endomorphisms $P_{XX}$ commute pairwise. The polarization
identity
\begin{equation}\label{eqn:4.58}
2P_{UX}=P_{U+X,U+X}-P_{UU}-P_{XX}
\end{equation}
expresses every $P_{UX}$ as a linear combination of the commuting endomorphisms
$P_{ZZ}$. Hence the entire family $\{P_{UX}\}$ is pairwise commuting.

Fix $p\in M^n$ and choose a common orthonormal eigenbasis $\{e_i\}$, with
the dual basis $\{\omega^i\}$. If $j\ne k$, then $P(U,X,e_j,e_k)=g(P_{UX}e_j,
e_k)=0$. The total symmetry of $P$ now shows that every component containing
two distinct indices vanishes. Consequently,
\begin{equation}\label{eqn:4.59}
P=\sum_{i=1}^np_i(\omega^i)^4,
\end{equation}
where $p_i=P(e_i,e_i,e_i,e_i)$. Moreover, we deduce from \eqref{eqn:2.25}
and \eqref{eqn:4.43} that
\begin{equation}\label{eqn:4.60}
\begin{aligned}
\sum_{j=1}^nP(U,X,e_j,e_j)
=\sum_{j=1}^ng((\nabla_UK)(X,e_j),e_j)
=ng(\nabla_UT,X)=n\rho g(U,X).
\end{aligned}
\end{equation}
On the other hand, tracing the preceding diagonal expression with
\eqref{eqn:4.59} gives
\begin{equation}\label{eqn:4.61}
\sum_{j=1}^nP(U,X,e_j,e_j)
=\sum_{i=1}^np_i\omega^i(U)\omega^i(X).
\end{equation}
Comparing the two expressions shows that $p_i=n\rho$ for every $i$, which
proves \eqref{eqn:4.56}.
\end{proof}

\subsection{Submanifolds with constant sectional curvature $c<0$}\label{sect:4.4}~

Throughout this subsection, let $M^n$ $(n\geq2)$ denote a Legendrian
submanifold in the unit sphere $\mathbb S^{2n+1}$ with constant sectional
curvature $c$ and conformal Maslov form. Recall the totally symmetric
tensor $P=\nabla K^\flat$ in \eqref{eqn:4.43}. For $U,X\in T_pM$, let
$P_{UX}\in\operatorname{End}(T_pM)$ be the self-adjoint endomorphism
determined by
\begin{equation}\label{eqn:4.62}
g(P_{UX}Y,Z)=P(U,X,Y,Z).
\end{equation}
For an orthonormal basis $\{e_a\}_{a=1}^n$, define
\begin{equation}\label{eqn:4.63}
R^P(X,Y)=\sum_{a=1}^n[P_{e_aX},P_{e_aY}].
\end{equation}
This definition is independent of the orthonormal basis. According to
the fact
\begin{equation}\label{eqn:4.64}
R(X,Y)Z=(X\wedge Y)Z+[K_X,K_Y]Z,
\end{equation}
where $(X\wedge Y)Z=g(Y,Z)X-g(X,Z)Y$, constant sectional curvature gives
\begin{equation}\label{eqn:4.65}
[K_X,K_Y]=\beta X\wedge Y,
\ \  \beta:=c-1.
\end{equation}
We now introduce the self-adjoint endomorphism
\begin{equation}\label{eqn:4.66}
\mathcal L:=K_T-\tfrac{n+1}{n}\beta I,
\ \
H_{\mathcal L}:=\tfrac1n\operatorname{trace}\mathcal L
=\|T\|^2-\tfrac{n+1}{n}\beta.
\end{equation}

\begin{lemma}\label{lem:4.4}
The tensor $P$ satisfies
\begin{equation}\label{eqn:4.67}
\sum_{i=1}^nP(U,X,e_i,e_i)=n\rho g(U,X),
\end{equation}
and
\begin{equation}\label{eqn:4.68}
[P_{UX},K_Y]+[K_X,P_{UY}]=0.
\end{equation}
Moreover, it holds that
\begin{equation}\label{eqn:4.69}
R^P(X,Y)
=\tfrac{nc}{2}\bigl(X\wedge\mathcal L Y
+\mathcal L X\wedge Y\bigr).
\end{equation}
\end{lemma}

\begin{proof}
Take the covariant derivative of $\sum_{i=1}^nK(e_i,e_i)=nT$ in a frame
that is normal at the point under consideration. By the total symmetry
of $P$ there holds
\begin{equation}\label{eqn:4.70}
\sum_{i=1}^nP(U,X,e_i,e_i)
=ng(\nabla_UT,X)=n\rho g(U,X),
\end{equation}
which proves \eqref{eqn:4.67}. Differentiating \eqref{eqn:4.65} and using
that $\beta$ is constant yields \eqref{eqn:4.68}.

To prove the last assertion, fix a point, take a normal orthonormal frame
there, and extend $X,Y$ such that their covariant derivatives vanish at that
point. By applying the connection Laplacian to \eqref{eqn:4.65} we obtain
\begin{equation}\label{eqn:4.71}
0=[(\Delta K)_X,K_Y]+[K_X,(\Delta K)_Y]
+2R^P(X,Y).
\end{equation}
With the help of \eqref{eqn:4.6}, it satisfies
\begin{equation}\label{eqn:4.72}
(\Delta K)(X,Y)=(n+1)cK(X,Y)-nc\Phi(X,Y),
\end{equation}
where
\begin{equation}\label{eqn:4.73}
\Phi(X,Y)=g(T,X)Y+g(T,Y)X+g(X,Y)T.
\end{equation}
Equivalently, $(\Delta K)_X=(n+1)cK_X-nc\Phi_X$. The total symmetry of
$K^\flat$ gives
\begin{equation}\label{eqn:4.74}
[\Phi_X,K_Y]+[K_X,\Phi_Y]
=X\wedge K_TY+K_TX\wedge Y.
\end{equation}
Substituting \eqref{eqn:4.65} and \eqref{eqn:4.74} into \eqref{eqn:4.71},
we conclude that
\begin{equation}\label{eqn:4.75}
\begin{aligned}
2R^P(X,Y)
&=nc\bigl(X\wedge K_TY+K_TX\wedge Y\bigr)
-2(n+1)c\beta X\wedge Y\\
&=nc\bigl(X\wedge\mathcal L Y
+\mathcal L X\wedge Y\bigr),
\end{aligned}
\end{equation}
which is \eqref{eqn:4.69}. Hence Lemma \ref{lem:4.4} has been proved.
\end{proof}

Define the tensor with values in skew-adjoint endomorphisms by
\begin{equation}\label{eqn:4.76}
\mathcal A(U,X,Y):=[P_{UX},K_Y].
\end{equation}
It is known from \eqref{eqn:4.68} and the symmetry of $P$ that $\mathcal A$
is symmetric in its three vector arguments.  We use the inner product $\langle
B,C\rangle=-\operatorname{tr}(BC)$ on skew-adjoint endomorphisms and define
the Gram endomorphism $G_{\mathcal A}$ by
\begin{equation}\label{eqn:4.77}
g(G_{\mathcal A}X,W)
=\sum_{a,i=1}^n
\langle\mathcal A(e_a,e_i,X),
\mathcal A(e_a,e_i,W)\rangle.
\end{equation}
In particular, $G_{\mathcal A}$ is positive semidefinite and $\operatorname{tr}
G_{\mathcal A}=\|\mathcal A\|^2$.

\begin{lemma}\label{lem:4.5}
One has
\begin{equation}\label{eqn:4.78}
G_{\mathcal A}
=-n\beta c\{(n-2)\mathcal L+nH_{\mathcal L}I\}.
\end{equation}
Consequently,
\begin{equation}\label{eqn:4.79}
\|\mathcal A\|^2
=-2n^2\beta c(n-1)H_{\mathcal L}.
\end{equation}
\end{lemma}

\begin{proof}
First observe that full symmetry and \eqref{eqn:4.67} yield
\begin{equation}\label{eqn:4.80}
\sum_{i=1}^n\mathcal A(U,e_i,e_i)
=[\sum_{i=1}^nP_{e_ie_i},K_U]
=[n\rho I,K_U]=0.
\end{equation}
Here \eqref{eqn:4.67} implies that $\sum_{i=1}^nP_{e_ie_i}=n\rho I$
can be regarded as an endomorphism. For the self-adjoint endomorphisms
$B,C,D,E$, invariance of the trace under commutators gives
\begin{equation}\label{eqn:4.81}
\langle[B,C],[D,E]\rangle
-\langle[B,E],[D,C]\rangle
=\langle[B,D],[C,E]\rangle.
\end{equation}
Apply this identity with $B=P_{e_ae_i}$, $C=K_X$, $D=P_{e_aX}$ and $E=K_{e_i}$,
and sum over $a,i$.  The term containing $[P_{e_ae_i},K_{e_i}]$ vanishes by
\eqref{eqn:4.80}. By using \eqref{eqn:4.65} we obtain
\begin{equation}\label{eqn:4.82}
g(G_{\mathcal A}X,X)
=\beta\sum_{i=1}^n
\langle R^P(e_i,X),X\wedge e_i\rangle.
\end{equation}
Based on \eqref{eqn:4.69}, the polarization and elementary identity
\begin{equation}\label{eqn:4.83}
\langle U\wedge V,X\wedge Y\rangle
=2\{g(U,X)g(V,Y)-g(U,Y)g(V,X)\}
\end{equation}
show that
\begin{equation}\label{eqn:4.84}
\sum_{i=1}^n\langle R^P(e_i,X),W\wedge e_i\rangle
=-nc\bigl\{(n-2)g(\mathcal LX,W)
+nH_{\mathcal L}g(X,W)\bigr\}.
\end{equation}
The right side is symmetric in $X,W$ and therefore \eqref{eqn:4.82} and
\eqref{eqn:4.84} prove \eqref{eqn:4.78}. Finally,
\begin{equation}\label{eqn:4.85}
\operatorname{tr}\{(n-2)\mathcal L+nH_{\mathcal L}I\}
=2n(n-1)H_{\mathcal L},
\end{equation}
and taking the trace of \eqref{eqn:4.78} gives \eqref{eqn:4.79}, which
completes the proof of Lemma \ref{lem:4.5}.
\end{proof}

\begin{proposition}\label{pro:4.3}
There is no Legendrian immersion of dimension $n\geq2$ in the unit
sphere that has conformal Maslov form and negative constant sectional
curvature.
\end{proposition}

\begin{proof}
Suppose that $c<0$. Then
\begin{equation}\label{eqn:4.86}
\beta=c-1<0,\ \
H_{\mathcal L}=\|T\|^2-\tfrac{n+1}{n}\beta>0.
\end{equation}
Thus $\beta c>0$, and the right side of \eqref{eqn:4.79} is strictly
negative. This contradicts $\|\mathcal A\|^2\geq0$. This pointwise
contradiction proves the assertion for every $n\geq2$.
\end{proof}

\subsection{Closed submanifolds and curvature pinching}\label{sect:4.5}~

In this subsection, we utilize the traceless tensor $\tilde K$ introduced
in \eqref{eqn:2.29}. Let $UM^n:=\{(p,v):p\in M^n,\ v\in T_pM^n,\ g(v,v)=1\}$
denote the unit tangent bundle of $M^n$. Its canonical measure $d\sigma$ is
obtained by integrating the standard spherical measure on each fiber $U_pM^n$
against the Riemannian volume measure on $M^n$.

\begin{proposition}\label{pro:4.4}
Let $M^n$ $(n\geq2)$ be a closed Legendrian submanifold in the
unit sphere $\mathbb S^{2n+1}$ with conformal Maslov form.
Then
\begin{equation}\label{eqn:4.87}
\int_{UM^n}\{
\|(\nabla_v\tilde K)(v,v)\|^2
+3g(R(\tilde K(v,v),v)v,\tilde K(v,v))
\}\,d\sigma=0.
\end{equation}
In particular, if its sectional curvature $\sec_g\geq0$, then $\nabla\tilde
K=0$.
\end{proposition}

\begin{proof}
By \eqref{eqn:2.31}, $\tilde K^\flat$ and $\nabla\tilde K^\flat$ are totally
symmetric and traceless. We apply the integral method of Ros and Urbano
\cite{Ros85,Urb86}, in the traceless form used in \cite[Section 3]{CHY19}.
We recall the calculation to fix the curvature sign. In this proof, let
$\phi:UM^n\to\mathbb R$ denote the scalar function $\phi(v):=\tilde
K^\flat(v,v,v)$. In an orthonormal frame $\{e_i\}_{i=1}^n$, we commute
the covariant derivatives by the Ricci identity:
\begin{equation}\label{eqn:4.88}
\begin{aligned}
(\Delta\tilde K^\flat)(v,v,v)
&=\sum_{i=1}^n(\nabla^2\tilde K^\flat)(e_i,e_i,v,v,v)
=\sum_{i=1}^n(\nabla^2\tilde K^\flat)(e_i,v,e_i,v,v)\\
&=-\sum_{i=1}^n\{g(R(e_i,v)e_i,\tilde K(v,v))
+2g(R(e_i,v)v,\tilde K(e_i,v))\},
\end{aligned}
\end{equation}
where $\operatorname{div}\tilde K^\flat=0$ follows from the full symmetry
and tracelessness of $\nabla\tilde K^\flat$. We then apply the unit
tangent bundle integration formula to the following tensor
\begin{equation}\label{eqn:4.89}
A(v_1,\ldots,v_7)
:=(\nabla_{v_1}\tilde K^\flat)(v_2,v_3,v_4)\,
\tilde K^\flat(v_5,v_6,v_7).
\end{equation}

Write $\nabla^H$ for the horizontal differentiation, with the fiber variable
extended by parallel translation. By the full symmetry of $\nabla\tilde
K^\flat$, the horizontal derivative can be written as
\begin{equation}\label{eqn:4.90}
\begin{aligned}
\nabla^H_{e_i}\phi(v)
&=(\nabla_{e_i}\tilde K^\flat)(v,v,v)
=(\nabla_v\tilde K^\flat)(v,v,e_i)\\
&=g((\nabla_v\tilde K)(v,v),e_i).
\end{aligned}
\end{equation}
The tensor $A$ therefore reads
\begin{equation}\label{eqn:4.91}
\begin{aligned}
A(e_i,v,v,v,v,v,v)
=\phi(v)\nabla^H_{e_i}\phi(v),\ \ 1\leq i\leq n.
\end{aligned}
\end{equation}
Squaring the components in \eqref{eqn:4.90} and summing, we find
\begin{equation}\label{eqn:4.92}
\|\nabla^H\phi(v)\|^2
=\sum_{i=1}^n(\nabla^H_{e_i}\phi(v))^2
=\sum_{i=1}^ng((\nabla_v\tilde K)(v,v),e_i)^2
=\|(\nabla_v\tilde K)(v,v)\|^2.
\end{equation}
Let $\Delta_H$ denote the horizontal Laplacian. Since the fiber variable
is extended by parallel translation, its action on $\phi$ takes the form
\begin{equation}\label{eqn:4.93}
\begin{aligned}
\Delta_H\phi(v)
&=\sum_{i=1}^n(\nabla^H_{e_i}\nabla^H_{e_i}\phi(v)
-\nabla^H_{\nabla_{e_i}e_i}\phi(v))\\
&=\sum_{i=1}^n(\nabla^2\tilde K^\flat)(e_i,e_i,v,v,v)
=(\Delta\tilde K^\flat)(v,v,v).
\end{aligned}
\end{equation}
Let $d\omega_p$ be the spherical measure on $U_pM^n$, which is preserved by
parallel translation. By \eqref{eqn:4.91}, the metric dual of $A(\cdot,v,v,v,
v,v,v)$ is $\phi(v)\nabla^H\phi(v)$, with horizontal vectors identified with
vectors in $T_pM^n$. Recall that $g(\nabla^H\phi,X)=\nabla^H_X\phi$ for $X\in
T_pM^n$. We apply the product rule and then differentiate this inner product
using metric compatibility:
\begin{equation}\label{eqn:4.94}
\begin{aligned}
&\sum_{i=1}^ng(\nabla^H_{e_i}(\phi\nabla^H\phi),e_i)
=\sum_{i=1}^n\{(\nabla^H_{e_i}\phi)g(\nabla^H\phi,e_i)
+\phi g(\nabla^H_{e_i}\nabla^H\phi,e_i)\}\\
&=\sum_{i=1}^n\{(\nabla^H_{e_i}\phi)^2
+\phi[\nabla^H_{e_i}g(\nabla^H\phi,e_i)
-g(\nabla^H\phi,\nabla_{e_i}e_i)]\}\\
&=\sum_{i=1}^n\{(\nabla^H_{e_i}\phi)^2
+\phi(\nabla^H_{e_i}\nabla^H_{e_i}\phi
-\nabla^H_{\nabla_{e_i}e_i}\phi)\}
=\|\nabla^H\phi\|^2+\phi\Delta_H\phi.
\end{aligned}
\end{equation}
Here $\nabla^H_{e_i}e_i=\nabla_{e_i}e_i$, since the frame is defined on
$M^n$.

Since parallel translation preserves $d\omega_p$, the covariant derivative
of a fiber integral is the fiber integral of the horizontal covariant derivative.
We now integrate the divergence over the closed manifold $M^n$. Writing the
iterated fiber integral with respect to $d\sigma$, we deduce that
\begin{equation}\label{eqn:4.95}
\begin{aligned}
0&=\int_{M^n}\operatorname{div}
(\int_{U_pM^n}\sum_{i=1}^nA(e_i,v,v,v,v,v,v)e_i\,d\omega_p(v))\,dV_g(p)\\
&=\int_{M^n}\int_{U_pM^n}
\sum_{i=1}^ng(\nabla^H_{e_i}(\phi\nabla^H\phi),e_i)
\,d\omega_p(v)\,dV_g(p)\\
&=\int_{UM^n}\{\|\nabla^H\phi(v)\|^2
+\phi(v)\Delta_H\phi(v)\}\,d\sigma.
\end{aligned}
\end{equation}
Substituting \eqref{eqn:4.88}, \eqref{eqn:4.92} and \eqref{eqn:4.93} into this
integral identity, we arrive at
\begin{equation}\label{eqn:4.96}
\begin{aligned}
&\int_{UM^n}\|(\nabla_v\tilde K)(v,v)\|^2\,d\sigma
=-\int_{UM^n}\phi(v)(\Delta\tilde K^\flat)(v,v,v)\,d\sigma\\
&=\int_{UM^n}\phi(v)\sum_{i=1}^n\{g(R(e_i,v)e_i,\tilde K(v,v))
+2g(R(e_i,v)v,\tilde K(e_i,v))\}\,d\sigma.
\end{aligned}
\end{equation}
On each unit tangent sphere $U_pM^n$, define
\begin{equation}\label{eqn:4.97}
\alpha_v(e):=\phi(v)\,g(R(e,v)v,\tilde K(v,v)),\ \
e\in T_v(U_pM^n).
\end{equation}
Fix $v\in U_pM^n$ and choose an orthonormal basis $\{e_1,\ldots,e_{n-1}\}$ of
$T_v(U_pM^n)$. Extend it to a local orthonormal frame $\{E_1,\ldots,E_{n-1}\}$
by parallel translation along the geodesics issuing from $v$. Let $\nabla^U$
and $\delta$ be the Levi-Civita connection and the codifferential on $U_pM^n$,
respectively. For $1\leq i\leq n-1$, the curve $\gamma_i(t)=v\cos t+e_i\sin t$
is a unit speed geodesic with $\gamma_i(0)=v$ and $\gamma_i'(0)=e_i$. Thus
$\gamma_i'(t)=-v\sin t+e_i\cos t=E_i(\gamma_i(t))$ and $\nabla^U_{\gamma_i'}
\gamma_i'=0$. The definition of the codifferential therefore reduces to
\begin{equation}\label{eqn:4.98}
\begin{aligned}
-\delta\alpha(v)
=\sum_{i=1}^{n-1}(\nabla^U_{E_i}\alpha)(E_i)\big|_v
=\sum_{i=1}^{n-1}E_i(\alpha(E_i))\big|_v,
\end{aligned}
\end{equation}
where we used $(\nabla^U_{E_i}E_i)|_v=0$. We evaluate these directional
derivatives along $\gamma_i$. Since $E_i(\gamma_i(t))=\gamma_i'(t)$,
differentiating \eqref{eqn:4.97} and setting $e_n=v$ gives
\begin{equation}\label{eqn:4.99}
\begin{aligned}
-\delta\alpha(v)
&=\sum_{i=1}^{n-1}\tfrac{d}{dt}
\{\alpha_{\gamma_i(t)}(\gamma_i'(t))\}\big|_{t=0}\\
&=3g(R(\tilde K(v,v),v)v,\tilde K(v,v))\\
&\quad+\phi(v)\sum_{i=1}^n\{g(R(e_i,v)e_i,\tilde K(v,v))
+2g(R(e_i,v)v,\tilde K(e_i,v))\}.
\end{aligned}
\end{equation}
On each closed fiber $U_pM^n$, the integral of $\delta\alpha$ is zero.
Integrating the preceding identity over the fibers and then over $M^n$,
and substituting this into \eqref{eqn:4.96}, we obtain \eqref{eqn:4.87}.
If $\sec_g\geq0$, both terms in the integrand of \eqref{eqn:4.87} are
nonnegative, and so $(\nabla_v\tilde K)(v,v)=0$ for every unit vector $v$.
By homogeneity this holds for every tangent vector. The full symmetry
of $\nabla\tilde K^\flat$ then allows polarization, which finally
proves $\nabla\tilde K=0$.
\end{proof}

\begin{lemma}\label{lem:4.6}
Let $M^n$ $(n\geq2)$ be a closed Legendrian submanifold in the unit
sphere $\mathbb S^{2n+1}$ with conformal Maslov form. If $\nabla\tilde
K=0$, then either $\tilde K=0$ or $\nabla K=0$. If, in addition, its
sectional curvature $\sec_g\leq1$, then $\nabla K=0$.
\end{lemma}

\begin{proof}
Suppose first that $\tilde K\ne0$. Since $\tilde K$ is parallel, it is
nowhere zero. Put $V:=\operatorname{grad}\rho$. We first compute the Lie
derivative of $\tilde K^\flat$ along $T$. By definition,
\begin{equation}\label{eqn:4.100}
\begin{aligned}
(\mathcal L_T\tilde K^\flat)(X,Y,Z)
={}&T(\tilde K^\flat(X,Y,Z))-\tilde K^\flat([T,X],Y,Z)\\
&-\tilde K^\flat(X,[T,Y],Z)-\tilde K^\flat(X,Y,[T,Z]).
\end{aligned}
\end{equation}
Since $[T,X]=\nabla_TX-\nabla_XT$, the terms involving $\nabla_TX$,
$\nabla_TY$ and $\nabla_TZ$ cancel. Applying $\nabla\tilde K^\flat=0$
and $\nabla_XT=\rho X$, we deduce that
\begin{equation}\label{eqn:4.101}
\begin{aligned}
(\mathcal L_T\tilde K^\flat)(X,Y,Z)
={}&(\nabla_T\tilde K^\flat)(X,Y,Z)
+\tilde K^\flat(\nabla_XT,Y,Z)\\
&+\tilde K^\flat(X,\nabla_YT,Z)
+\tilde K^\flat(X,Y,\nabla_ZT)
=3\rho\tilde K^\flat(X,Y,Z).
\end{aligned}
\end{equation}
For the connection $\nabla$, its Lie derivative is the $(1,2)$-tensor
defined by the first line below. Expanding the Lie brackets and using
\eqref{eqn:2.11}, we find (cf. \cite[p. 9]{Yan57})
\begin{equation}\label{eqn:4.102}
\begin{aligned}
(\mathcal L_T\nabla)_XY
&=[T,\nabla_XY]-\nabla_{[T,X]}Y-\nabla_X[T,Y]\\
&=R(T,X)Y+\nabla_X\nabla_YT-\nabla_{\nabla_XY}T\\
&=R(T,X)Y+(X\rho)Y.
\end{aligned}
\end{equation}
To evaluate the curvature term, we again use $\nabla T=\rho I$ to obtain
\begin{equation}\label{eqn:4.103}
\begin{aligned}
R(Y,Z)T
&=\nabla_Y(\rho Z)-\nabla_Z(\rho Y)-\rho[Y,Z]
=(Y\rho)Z-(Z\rho)Y,\\
g(R(T,X)Y,Z)
&=g(R(Y,Z)T,X)
=(Y\rho)g(X,Z)-g(X,Y)g(V,Z).
\end{aligned}
\end{equation}
Since $Z$ is arbitrary, $R(T,X)Y=(Y\rho)X-g(X,Y)V$. Substitution into
\eqref{eqn:4.102} gives
\begin{equation}\label{eqn:4.104}
(\mathcal L_T\nabla)_XY
=(X\rho)Y+(Y\rho)X-g(X,Y)V.
\end{equation}

We now compute the Lie derivative of $\nabla\tilde K^\flat=0$ along $T$
as follows (cf. \cite[p. 16]{Yan57}):
\begin{equation}\label{eqn:4.105}
\begin{aligned}
0=~&(\mathcal L_T(\nabla\tilde K^\flat))(X,Y,Z,W)\\
=~&(\nabla_X(\mathcal L_T\tilde K^\flat))(Y,Z,W)
-\tilde K^\flat((\mathcal L_T\nabla)_XY,Z,W)\\
~&-\tilde K^\flat(Y,(\mathcal L_T\nabla)_XZ,W)
-\tilde K^\flat(Y,Z,(\mathcal L_T\nabla)_XW).
\end{aligned}
\end{equation}
By \eqref{eqn:4.101} and $\nabla\tilde K^\flat=0$,
\begin{equation}\label{eqn:4.106}
\begin{aligned}
\nabla_X(\mathcal L_T\tilde K^\flat)
=3(X\rho)\tilde K^\flat+3\rho\nabla_X\tilde K^\flat
=3(X\rho)\tilde K^\flat.
\end{aligned}
\end{equation}
Expanding the three connection terms by \eqref{eqn:4.104}, we calculate
that
\begin{equation}\label{eqn:4.107}
\begin{aligned}
0={}&3(X\rho)\tilde K^\flat(Y,Z,W)\\
&-\{(X\rho)\tilde K^\flat(Y,Z,W)
+(Y\rho)\tilde K^\flat(X,Z,W)
-g(X,Y)\tilde K^\flat(V,Z,W)\}\\
&-\{(X\rho)\tilde K^\flat(Y,Z,W)
+(Z\rho)\tilde K^\flat(Y,X,W)
-g(X,Z)\tilde K^\flat(Y,V,W)\}\\
&-\{(X\rho)\tilde K^\flat(Y,Z,W)
+(W\rho)\tilde K^\flat(Y,Z,X)
-g(X,W)\tilde K^\flat(Y,Z,V)\},
\end{aligned}
\end{equation}
and hence
\begin{equation}\label{eqn:4.108}
\begin{aligned}
&(Y\rho)\tilde K^\flat(X,Z,W)+(Z\rho)\tilde K^\flat(X,Y,W)
+(W\rho)\tilde K^\flat(X,Y,Z)\\
&=g(X,Y)\tilde K^\flat(V,Z,W)+g(X,Z)\tilde K^\flat(V,Y,W)
+g(X,W)\tilde K^\flat(V,Y,Z).
\end{aligned}
\end{equation}
Set $X=Y=e_i$ and sum over $1\leq i\leq n$. Since $\tilde K^\flat$ is
traceless,
\begin{equation}\label{eqn:4.109}
\begin{aligned}
\tilde K^\flat(V,Z,W)
=(n+2)\tilde K^\flat(V,Z,W).
\end{aligned}
\end{equation}
Thus $\tilde K^\flat(V,\cdot,\cdot)=0$. Substituting $Y=V$ into
\eqref{eqn:4.108} now gives $\|V\|^2\tilde K^\flat(X,Z,W)=0$,
so $V=0$. The constant $\rho$ vanishes by the divergence theorem:
\begin{equation}\label{eqn:4.110}
0=\int_{M^n}\operatorname{div}T\,dV
=n\rho\operatorname{Vol}(M^n).
\end{equation}
Therefore $\nabla T=0$ and $\nabla K=0$ by \eqref{eqn:2.31}.

If $\tilde K=0$, substituting \eqref{eqn:2.29} into the Gauss equation
gives the following identity for orthonormal $X,Y$:
\begin{equation}\label{eqn:4.111}
\sec_g(X,Y)
=1+\tfrac{n^2}{(n+2)^2}
\{\|T\|^2+g(T,X)^2+g(T,Y)^2\}.
\end{equation}
The upper bound $\sec_g\leq1$ forces $T=0$ and hence $K=0$. Thus the
additional upper bound also gives $\nabla K=0$ in the remaining case.
Hence Lemma \ref{lem:4.6} has been proved.
\end{proof}

\begin{lemma}\label{lem:4.7}
Let $M^n$ $(n\geq2)$ be a Legendrian submanifold in the unit sphere
$\mathbb S^{2n+1}$ with $C$-parallel second fundamental form. If its
sectional curvature satisfies $\sec_g\leq1$, then $M^n$ is either
totally geodesic or flat.
\end{lemma}

\begin{proof}
With the help of \eqref{eqn:2.28}, it is known that the $C$-parallel condition
is equivalent to $\nabla K=0$. The Gauss equation \eqref{eqn:2.26} reads
\begin{equation}\label{eqn:4.112}
R(X,Y)Z=(X\wedge Y)Z+[K_X,K_Y]Z,
\end{equation}
where $(X\wedge Y)Z=g(Y,Z)X-g(X,Z)Y$.
If $K=0$, then $h=\varphi K=0$ and $M^n$ is totally geodesic. Otherwise,
$K$ is nowhere zero since it is parallel. At a fixed point $p$, choose a
unit vector $e_1$ maximizing $g(K(u,u),u)$ on the unit tangent sphere. Its
maximum $\mu_1$ is positive, since the cubic is odd and nonzero. The first
and second derivative tests give an orthonormal basis $\{e_i\}_{i=1}^n$
such that
\begin{equation}\label{eqn:4.113}
K(e_1,e_1)=\mu_1e_1,\ \
K(e_1,e_i)=\mu_i e_i,\ \
2\mu_i\leq\mu_1,\ \ 2\leq i\leq n.
\end{equation}
Indeed, one diagonalizes the self-adjoint map $K_{e_1}$ on $e_1^\perp$
and differentiates the cubic along $e_1\cos t+e_i\sin t$.

For $2\leq i\leq n$, put $k_i:=\sec_g(e_1,e_i)$. Substitution of \eqref{eqn:4.113}
into \eqref{eqn:4.112}, followed by the Ricci identity $R\cdot K=0$, gives
\begin{equation}\label{eqn:4.114}
\begin{aligned}
0&=(R(e_1,e_i)\cdot K)(e_1,e_1)\\
&=R(e_1,e_i)K(e_1,e_1)-2K(R(e_1,e_i)e_1,e_1)\\
&=(2\mu_i-\mu_1)k_i e_i,
\end{aligned}
\end{equation}
where $R(e_1,e_i)e_1=-(1+\mu_1\mu_i-\mu_i^2)e_i=-k_i e_i$.
If $\mu_i=\mu_1/2$, then $k_i=1+\mu_1^2/4>1$, contrary to the curvature
bound. Hence $k_i=0$, and $2\mu_i\leq\mu_1$ selects the smaller root:
\begin{equation}\label{eqn:4.115}
\mu_2=\cdots=\mu_n
=a:=\tfrac12(\mu_1-\sqrt{\mu_1^2+4})<0.
\end{equation}
For arbitrary $X,Y\in T_pM^n$, skew-symmetry gives
$R(X,Y)e_1\perp e_1$. Since $K(e_1,Z)=aZ$ for $Z\perp e_1$, the
identity $R\cdot K=0$ further implies
\begin{equation}\label{eqn:4.116}
\begin{aligned}
0&=(R(X,Y)\cdot K)(e_1,e_1)\\
&=(\mu_1I-2K_{e_1})R(X,Y)e_1
=\sqrt{\mu_1^2+4}\,R(X,Y)e_1.
\end{aligned}
\end{equation}
Thus $R(X,Y)e_1=0$, and the curvature symmetries give
\begin{equation}\label{eqn:4.117}
g(R(e_1,X)Y,Z)=g(R(Y,Z)e_1,X)=0.
\end{equation}
Hence every curvature component involving $e_1$ vanishes.

Let $\pi_1^\perp Z=Z-g(Z,e_1)e_1$ denote the orthogonal projection onto
$V_1=e_1^\perp$.
For $X,Y\in T_pM^n$ and $Z\in V_1$, we have $g(R(X,Y)Z,e_1)=
-g(Z,R(X,Y)e_1)=0$, so
\begin{equation}\label{eqn:4.118}
R(X,Y)Z\in V_1.
\end{equation}
For arbitrary $Z\in T_pM^n$, decomposing $Z=g(Z,e_1)e_1+\pi_1^\perp Z$
immediately gives
\begin{equation}\label{eqn:4.119}
\pi_1^\perp R(X,Y)Z=R(X,Y)Z=R(X,Y)\pi_1^\perp Z.
\end{equation}
Denote the restriction of $R$ to $V_1$ by $R^1$ and set $K^1:=\pi_1^\perp
\circ K|_{V_1\times V_1}$. For $X,Y\in V_1$, it follows that
\begin{equation}\label{eqn:4.120}
\begin{aligned}
K(X,Y)=a g(X,Y)e_1+K^1(X,Y).
\end{aligned}
\end{equation}
For $X,Y,Z\in V_1$, the cubic symmetry cancels the $e_1$ components in
$K_XK_YZ-K_YK_XZ$, and thus
\begin{equation}\label{eqn:4.121}
\relax[K_X,K_Y]Z=a^2(X\wedge Y)Z+[K^1_X,K^1_Y]Z.
\end{equation}
The restricted Gauss equation becomes
\begin{equation}\label{eqn:4.122}
R^1(X,Y)Z=(1+a^2)(X\wedge Y)Z+[K^1_X,K^1_Y]Z.
\end{equation}
According to \eqref{eqn:4.118} and \eqref{eqn:4.119}, for $X,Y,Z,W\in V_1$,
we calculate that
\begin{equation}\label{eqn:4.123}
\begin{aligned}
&(R^1(X,Y)\cdot K^1)(Z,W)\\
&=R(X,Y)\pi_1^\perp K(Z,W)-\pi_1^\perp K(R(X,Y)Z,W)
-\pi_1^\perp K(Z,R(X,Y)W)\\
&=\pi_1^\perp\bigl(R(X,Y)K(Z,W)-K(R(X,Y)Z,W)
-K(Z,R(X,Y)W)\bigr)\\
&=\pi_1^\perp((R(X,Y)\cdot K)(Z,W))=0.
\end{aligned}
\end{equation}
The cubic form $(K^1)^\flat$ is totally symmetric, and the sectional
curvature of $R^1$ is at most $1$. If $\dim V_1\geq2$, then $K^1\ne0$,
since otherwise \eqref{eqn:4.122} gives sectional curvature
$1+a^2>1$.

We proceed by induction on the dimension of
$V_r=\operatorname{span}\{e_1,\ldots,e_r\}^\perp$. Suppose that the
restricted curvature tensor $R^r$ and projected tensor $K^r$ satisfy
\begin{equation}\label{eqn:4.124}
R^r(X,Y)Z=c_r(X\wedge Y)Z+[K^r_X,K^r_Y]Z,\ \ R^r\cdot K^r=0,
\end{equation}
where $c_r>1$, the sectional curvature of $R^r$ is at most $1$, and
$(K^r)^\flat$ is totally symmetric. These conditions hold for $r=1$
with $c_1=1+a^2$. If $\dim V_r\geq2$, the curvature bound excludes
$K^r=0$. Choose a unit vector $e_{r+1}$ maximizing its cubic, with
maximum $\nu_r>0$. As in \eqref{eqn:4.113}, an orthonormal basis of
$V_r$ can be chosen such that
\begin{equation}\label{eqn:4.125}
\begin{aligned}
K^r(e_{r+1},e_{r+1})=\nu_r e_{r+1},\ \
K^r(e_{r+1},e_j)=\beta_j e_j,\ \
2\beta_j\leq\nu_r,
\end{aligned}
\end{equation}
where $r+2\leq j\leq n$. Put $k_j:=g(R^r(e_{r+1},e_j)e_j,e_{r+1})$.
Applying \eqref{eqn:4.124} and evaluating $R^r(e_{r+1},e_j)\cdot K^r$
on $(e_{r+1},e_{r+1})$, we find
\begin{equation}\label{eqn:4.126}
k_j=c_r+\nu_r\beta_j-\beta_j^2\leq1,\ \
(2\beta_j-\nu_r)k_j=0.
\end{equation}
The value $\beta_j=\nu_r/2$ guarantees that $k_j=c_r+\nu_r^2/4>1$.
Hence $k_j=0$, and the smaller root is selected by
$2\beta_j\leq\nu_r$:
\begin{equation}\label{eqn:4.127}
\beta_{r+2}=\cdots=\beta_n
=a_{r+1}:=\tfrac12(\nu_r-\sqrt{\nu_r^2+4c_r})<0.
\end{equation}
For arbitrary $X,Y\in V_r$, the argument in \eqref{eqn:4.116} now yields
\begin{equation}\label{eqn:4.128}
\begin{aligned}
0&=(R^r(X,Y)\cdot K^r)(e_{r+1},e_{r+1})\\
&=(\nu_r-2a_{r+1})R^r(X,Y)e_{r+1}
=\sqrt{\nu_r^2+4c_r}\,R^r(X,Y)e_{r+1}.
\end{aligned}
\end{equation}
Thus every curvature component involving $e_{r+1}$ vanishes.
Restricting to $V_{r+1}=V_r\cap e_{r+1}^\perp$ and projecting $K^r$
as in \eqref{eqn:4.120}--\eqref{eqn:4.123}, we obtain
\begin{equation}\label{eqn:4.129}
\begin{aligned}
&R^{r+1}(X,Y)Z=c_{r+1}(X\wedge Y)Z
+[K^{r+1}_X,K^{r+1}_Y]Z,\\
&c_{r+1}=c_r+a_{r+1}^2>c_r>1,\ \ R^{r+1}\cdot K^{r+1}=0.
\end{aligned}
\end{equation}
The cubic symmetry and the curvature upper bound are also preserved,
completing the induction step. On the final one dimensional space
$V_{n-1}$, we conclude that the curvature tensor $R^{n-1}$ vanishes.
All components involving the preceding maximizing directions have
already vanished and thus $R_p=0$. Since $p$ was arbitrary, $M^n$
is flat. This completes the proof of Lemma \ref{lem:4.7}.
\end{proof}

\section{Proofs of the main theorems}\label{sect:5}

We combine the properties established in Section \ref{sect:4} with the
calculations in Section \ref{sect:3} to complete the proofs of the two
main theorems. The following theorem identifies the flat Legendrian
submanifolds in $\mathbb S^{2n+1}$ with $C$-parallel second fundamental
form.

\begin{theorem}\label{thm:5.1}
Let $M^n$ $(n\geq2)$ be a flat Legendrian submanifold in the unit
sphere $\mathbb S^{2n+1}$ with $C$-parallel second fundamental form.
Then $M^n$ is locally congruent to a weighted Clifford Legendrian
immersion $\Psi_\alpha$ given in Example \ref{exa:1.2}.
\end{theorem}

\begin{proof}
According to \eqref{eqn:2.28}, the $C$-parallel condition is equivalent to
$\nabla K=0$. Fixing a point $x\in M^n$, we know from the Gauss equation
\eqref{eqn:2.26} that $K(x)\ne0$. Let $U_xM^n$ be the unit sphere in $T_x
M^n$ and define $f_1(u):=g(K(u,u),u)$ on $U_xM^n$. Since the cubic symmetric
form determines $K$ by polarization, the function $f_1$ is odd and not
identically zero. Therefore, it attains a positive maximum $\lambda_1$
at a unit vector $e_1$. The first and second variation formulas, followed
by diagonalization of the self-adjoint map $K_{e_1}$ on $e_1^\perp$, give
an orthonormal basis $\{e_i\}_{i=1}^n$ such that
\begin{equation}\label{eqn:5.1}
K(e_1,e_1)=\lambda_1e_1,\ \
K(e_1,e_j)=\lambda_{1j}e_j,\ \
\lambda_1\geq2\lambda_{1j},\ \ 2\leq j\leq n,
\end{equation}
where $\lambda_{1j}$ are the eigenvalues of $K_{e_1}$.
On each plane spanned by $e_1,e_j$ with $j\geq2$, the Gauss equation reads
\begin{equation}\label{eqn:5.2}
0=g(R(e_1,e_j)e_j,e_1)
=1+\lambda_1\lambda_{1j}-\lambda_{1j}^2.
\end{equation}
The inequality in \eqref{eqn:5.1} selects the smaller root. Consequently,
\begin{equation}\label{eqn:5.3}
\lambda_{12}=\cdots=\lambda_{1n}
=\tfrac12\bigl(\lambda_1-\sqrt{\lambda_1^2+4}\,\bigr)
=:\mu_1<0.
\end{equation}

We next repeat this construction on successive orthogonal complements.
Suppose that $2\leq k\leq n-1$ and that $e_1,\ldots,e_{k-1}$ have been
chosen. Put $T_x^kM^n:=\operatorname{span}\{e_1,\ldots,e_{k-1}\}^\perp$,
\begin{equation}\label{eqn:5.4}
K^{(k)}(X,Y):=\pi_{k-1}^\perp K(X,Y),\ \ X,Y\in T_x^kM^n,
\end{equation}
where $\pi_{k-1}^\perp X:=X-\sum_{j=1}^{k-1}g(X,e_j)e_j$ is the orthogonal
projection onto $T_x^kM^n$. It is obvious that the cubic form $g(K^{(k)}(X,Y),
Z)$ is totally symmetric. Choose a unit vector $e_k\in T_x^kM^n$ at which
$g(K^{(k)}(u,u),u)$ attains its maximum $\lambda_k$. The first and second
variations show that
\begin{equation}\label{eqn:5.5}
g(K^{(k)}(e_k,e_k),v)=0,\ \
2g(K^{(k)}(e_k,v),v)\leq\lambda_k
\end{equation}
for every unit vector $v\in T_x^kM^n$ orthogonal to $e_k$. Diagonalizing
$K^{(k)}_{e_k}$ on the orthogonal complement of $e_k$ and using the total
symmetry, we obtain
\begin{equation}\label{eqn:5.6}
\begin{aligned}
K(e_k,e_k)=\sum_{m=1}^{k-1}\mu_me_m+\lambda_ke_k,\ \
K(e_k,e_j)=\lambda_{kj}e_j,\ \ j>k.
\end{aligned}
\end{equation}
The Gauss equation on the plane spanned by $e_k,e_j$ now becomes
\begin{equation}\label{eqn:5.7}
0=1+\sum_{m=1}^{k-1}\mu_m^2
+\lambda_k\lambda_{kj}-\lambda_{kj}^2,\ \ j>k.
\end{equation}
The smaller root of \eqref{eqn:5.7} is selected by \eqref{eqn:5.5},
that is,
\begin{equation}\label{eqn:5.8}
\lambda_{kj}
=\tfrac12\big(\lambda_k-
\sqrt{\lambda_k^2+4\bigl(1+\mu_1^2+\cdots+\mu_{k-1}^2\bigr)}\,\big)
=:\mu_k<0,\ \ j>k.
\end{equation}
Moreover, $\lambda_k>0$. Indeed, if $\lambda_k=0$, the odd cubic $g(K^{(k)}(u,
u),u)$ vanishes identically. Polarization then gives $K^{(k)}=0$, contradicting
\eqref{eqn:5.7}. Choose a unit vector $e_n$ in the remaining one dimensional
orthogonal complement and put $\lambda_n:=g(K(e_n,e_n),e_n)$. For $n=2$, the
total symmetry directly yields $K(e_2,e_2)=\mu_1e_1+\lambda_2e_2$. We then
obtain
\begin{equation}\label{eqn:5.9}
\begin{cases}
K(e_1,e_1)=\lambda_1e_1,\ \ K(e_1,e_k)=\mu_1e_k,\\
K(e_k,e_k)=\mu_1e_1+\cdots+\mu_{k-1}e_{k-1}+\lambda_ke_k,\\
K(e_i,e_j)=\mu_i e_j,\ \ 2\leq i<j\leq n,\ \ 2\leq k\leq n,
\end{cases}
\end{equation}
where, after setting $q_k:=1+\sum_{j=1}^{k-1}\mu_j^2$, the coefficients
satisfy
\begin{equation}\label{eqn:5.10}
\mu_k(\lambda_k-\mu_k)=-q_k,\ \
\lambda_k>0,\ \ \mu_k<0,\ \ 1\leq k\leq n-1.
\end{equation}

After shrinking to a simply connected flat neighborhood of $x$, we extend $\{e_i
\}_{i=1}^n$ by path independent parallel translation to a parallel orthonormal
frame $\{E_i\}_{i=1}^n$. Since $K$ is parallel, its algebraic normal form is
preserved and all coefficients in \eqref{eqn:5.9} are constant. In this situation,
it holds that
\begin{equation}\label{eqn:5.11}
\nabla_{E_i}E_j=0,\ \ 1\leq i,j\leq n.
\end{equation}
Introduce flat coordinates $(w_1,\ldots,w_n)$ with $E_i:=\partial/\partial w_i$.
Let $F$ denote the Legendrian immersion of this neighborhood into $\mathbb S^{2n+1}
\subset\mathbb C^{n+1}$. Setting $F_{w_j}:=\partial F/\partial w_j$ and $F_{w_jw_k}
:=\partial^2F/\partial w_j\partial w_k$, by applying the Gauss formula we see from
\eqref{eqn:5.9} and \eqref{eqn:5.11} that
\begin{equation}\label{eqn:5.12}
\begin{aligned}
&F_{w_jw_k}=\mu_j\mathrm{i}F_{w_k},\ \ 1\leq j<k\leq n,\\
&F_{w_\ell w_\ell}
=\sum_{j=1}^{\ell-1}\mu_j\mathrm{i}F_{w_j}
+\lambda_\ell\mathrm{i}F_{w_\ell}-F,\ \ 1\leq\ell\leq n.
\end{aligned}
\end{equation}

Since $\mu_1(\lambda_1-\mu_1)=-1$, the characteristic roots of the equation for
$F_{w_1w_1}$ are $\mathrm{i}\mu_1$ and $\mathrm{i}(\lambda_1-\mu_1)$. Hence
\begin{equation}\label{eqn:5.13}
F=\widetilde B_1(w_2,\ldots,w_n)
e^{\mathrm{i}(\lambda_1-\mu_1)w_1}
+B_2(w_2,\ldots,w_n)e^{\mathrm{i}\mu_1w_1}.
\end{equation}
Substituting \eqref{eqn:5.13} into the mixed equations gives
\begin{equation}\label{eqn:5.14}
0=F_{w_1w_k}-\mu_1\mathrm{i}F_{w_k}
=\mathrm{i}(\lambda_1-2\mu_1)
\tfrac{\partial\widetilde B_1}{\partial w_k}
e^{\mathrm{i}(\lambda_1-\mu_1)w_1},\ \ 2\leq k\leq n.
\end{equation}
Since $\lambda_1-2\mu_1=\sqrt{\lambda_1^2+4}>0$, we have $\partial\widetilde B_1
/\partial w_k=0$ for $2\leq k\leq n$. As a result, $\widetilde B_1$ is a constant
vector, which we denote by $A_1$, and
\begin{equation}\label{eqn:5.15}
F=A_1e^{\mathrm{i}(\lambda_1-\mu_1)w_1}
+B_2(w_2,\ldots,w_n)e^{\mathrm{i}\mu_1w_1}.
\end{equation}
We apply the Euclidean inner product $\langle\cdot,\cdot\rangle$ on $\mathbb
C^{n+1}\cong\mathbb R^{2n+2}$. Expanding $\langle F,F\rangle=1$ in \eqref{eqn:5.15},
we compute
\begin{equation}\label{eqn:5.16}
\begin{aligned}
1={}&\|A_1\|^2+\|B_2\|^2
+2\cos\bigl((\lambda_1-2\mu_1)w_1\bigr)\langle A_1,B_2\rangle\\
&-2\sin\bigl((\lambda_1-2\mu_1)w_1\bigr)
\langle A_1,\mathrm{i}B_2\rangle.
\end{aligned}
\end{equation}
For fixed $w_2,\ldots,w_n$, the vectors $A_1$ and $B_2$ are independent
of $w_1$. Since $\lambda_1-2\mu_1\ne0$, the coefficients of the sine and
cosine terms must vanish. Hence
\begin{equation}\label{eqn:5.17}
\langle A_1,B_2\rangle
=\langle A_1,\mathrm{i}B_2\rangle=0,\ \
\|A_1\|^2+\|B_2\|^2=1.
\end{equation}
Consequently, by the Legendrian condition, it can be checked that
\begin{equation}\label{eqn:5.18}
\begin{aligned}
0=\langle F_{w_1},\mathrm{i}F\rangle
=(\lambda_1-\mu_1)\|A_1\|^2+\mu_1\|B_2\|^2.
\end{aligned}
\end{equation}
Note from $\mu_1(\lambda_1-\mu_1)=-1$ that this system has the following solution:
\begin{equation}\label{eqn:5.19}
\|A_1\|^2=\tfrac{\mu_1^2}{q_1q_2},\ \
\|B_2\|^2=\tfrac{1}{q_2}.
\end{equation}

We now proceed by induction. Suppose that $2\leq k\leq n-1$ and
\begin{equation}\label{eqn:5.20}
\begin{aligned}
F=\sum_{\alpha=1}^{k-1}
A_\alpha e^{\mathrm{i}\sum_{j=1}^{\alpha-1}\mu_jw_j
+\mathrm{i}(\lambda_\alpha-\mu_\alpha)w_\alpha}
+B_k(w_k,\ldots,w_n)e^{\mathrm{i}\sum_{j=1}^{k-1}\mu_jw_j}.
\end{aligned}
\end{equation}
Assume also that $A_1,\ldots,A_{k-1},B_k$ and their images under multiplication
by $\mathrm{i}$ are mutually orthogonal and satisfy
\begin{equation}\label{eqn:5.21}
\|A_\alpha\|^2
=\tfrac{\mu_\alpha^2}{q_\alpha q_{\alpha+1}},
\ \ 1\leq\alpha\leq k-1,\ \
\|B_k\|^2=\tfrac{1}{q_k}.
\end{equation}
The equation for $F_{w_kw_k}$ in \eqref{eqn:5.12} can be written as
\begin{equation}\label{eqn:5.22}
\tfrac{\partial^2B_k}{\partial w_k^2}
-\lambda_k\mathrm{i}\tfrac{\partial B_k}{\partial w_k}
+q_kB_k=0.
\end{equation}
Based on \eqref{eqn:5.10}, its characteristic roots are $\mathrm{i}\mu_k$
and $\mathrm{i}(\lambda_k-\mu_k)$. Therefore,
\begin{equation}\label{eqn:5.23}
\begin{aligned}
B_k(w_k,\ldots,w_n)
=\widetilde B_k(w_{k+1},\ldots,w_n)
e^{\mathrm{i}(\lambda_k-\mu_k)w_k}
+B_{k+1}(w_{k+1},\ldots,w_n)e^{\mathrm{i}\mu_kw_k},
\end{aligned}
\end{equation}
and the mixed equations $F_{w_kw_\ell}=\mu_k\mathrm{i}F_{w_\ell}$ yield
\begin{equation}\label{eqn:5.24}
\begin{aligned}
0=\tfrac{\partial^2B_k}{\partial w_k\partial w_\ell}
-\mu_k\mathrm{i}\tfrac{\partial B_k}{\partial w_\ell}
=\mathrm{i}(\lambda_k-2\mu_k)
\tfrac{\partial\widetilde B_k}{\partial w_\ell}
e^{\mathrm{i}(\lambda_k-\mu_k)w_k},\ \ k<\ell\leq n.
\end{aligned}
\end{equation}
Since $\lambda_k-2\mu_k=\sqrt{\lambda_k^2+4q_k}>0$, it follows that
$\partial\widetilde B_k/\partial w_\ell=0$ for $k<\ell\leq n$.
Thus $\widetilde B_k$ is a constant vector, say $A_k$, and
\eqref{eqn:5.20} becomes
\begin{equation}\label{eqn:5.25}
\begin{aligned}
F=\sum_{\alpha=1}^{k}
A_\alpha e^{\mathrm{i}\sum_{j=1}^{\alpha-1}\mu_jw_j
+\mathrm{i}(\lambda_\alpha-\mu_\alpha)w_\alpha}
+B_{k+1}(w_{k+1},\ldots,w_n)
e^{\mathrm{i}\sum_{j=1}^{k}\mu_jw_j}.
\end{aligned}
\end{equation}
Expanding the norm $\|B_k\|^2=1/q_k$ in \eqref{eqn:5.23} as in \eqref{eqn:5.16},
and using $\lambda_k-2\mu_k\ne0$, we get
\begin{equation}\label{eqn:5.26}
\langle A_k,B_{k+1}\rangle
=\langle A_k,\mathrm{i}B_{k+1}\rangle=0.
\end{equation}
The norm identity and the Legendrian condition therefore give
\begin{equation}\label{eqn:5.27}
\begin{aligned}
\|A_k\|^2+\|B_{k+1}\|^2=\tfrac{1}{q_k},\ \
(\lambda_k-\mu_k)\|A_k\|^2+\mu_k\|B_{k+1}\|^2=0.
\end{aligned}
\end{equation}
The solution of \eqref{eqn:5.27} can be written in terms of $q_{k+1}=q_k
+\mu_k^2$ as
\begin{equation}\label{eqn:5.28}
\|A_k\|^2=\tfrac{\mu_k^2}{q_kq_{k+1}},\ \
\|B_{k+1}\|^2=\tfrac{1}{q_{k+1}}.
\end{equation}
For $1\leq\alpha<k$, the induction hypothesis and the constancy of $A_\alpha$
further imply that both $B_k$ and $\partial B_k/\partial w_k$ are orthogonal
to $A_\alpha$ and $\mathrm{i}A_\alpha$. Solving \eqref{eqn:5.23}, we have
\begin{equation}\label{eqn:5.29}
\begin{aligned}
&A_ke^{\mathrm{i}(\lambda_k-\mu_k)w_k}
=\tfrac{1}{\mathrm{i}(\lambda_k-2\mu_k)}
\bigl(\tfrac{\partial B_k}{\partial w_k}-\mathrm{i}\mu_kB_k\bigr),\\
&B_{k+1}e^{\mathrm{i}\mu_kw_k}
=\tfrac{1}{\mathrm{i}(\lambda_k-2\mu_k)}
\bigl(\mathrm{i}(\lambda_k-\mu_k)B_k
-\tfrac{\partial B_k}{\partial w_k}\bigr).
\end{aligned}
\end{equation}
Since $\lambda_k-2\mu_k\ne0$, both vectors are complex linear combinations
of $B_k$ and its derivative. Thus $A_k$ and $B_{k+1}$ are Hermitian orthogonal
to every $A_\alpha$ with $\alpha<k$. Together with \eqref{eqn:5.26}, this proves
that $A_1,\ldots,A_k,B_{k+1}$ and their images under multiplication by $\mathrm{i}$
are mutually orthogonal, completing the induction.

For $n=2$, the next formula follows directly from \eqref{eqn:5.15}. For $n\geq3$,
take $k=n-1$ in \eqref{eqn:5.25}. The remaining function $B_n(w_n)$ satisfies
\begin{equation}\label{eqn:5.30}
B_n''-\lambda_n\mathrm{i}B_n'+q_nB_n=0.
\end{equation}
Set
\begin{equation}\label{eqn:5.31}
\delta=\sqrt{\lambda_n^2+4q_n},\ \
r_\pm=\tfrac12(\lambda_n\pm\delta).
\end{equation}
Then $r_+>0>r_-$ and $r_+r_-=-q_n$. Solving \eqref{eqn:5.30}, we obtain
\begin{equation}\label{eqn:5.32}
B_n(w_n)=A_ne^{\mathrm{i}r_+w_n}
+A_{n+1}e^{\mathrm{i}r_-w_n},
\end{equation}
where $A_n,A_{n+1}\in\mathbb C^{n+1}$ are constant vectors. Since $\|B_n\|^2=1/q_n$
is constant and $r_+\ne r_-$, taking real and imaginary parts shows that $A_n$ and
$A_{n+1}$ are Hermitian orthogonal. Their squared norms satisfy the normalization
and Legendrian conditions:
\begin{equation}\label{eqn:5.33}
\|A_n\|^2+\|A_{n+1}\|^2=\tfrac{1}{q_n},\ \
r_+\|A_n\|^2+r_-\|A_{n+1}\|^2=0.
\end{equation}
Consequently,
\begin{equation}\label{eqn:5.34}
\|A_n\|^2=\tfrac{-r_-}{q_n(r_+-r_-)},\ \
\|A_{n+1}\|^2=\tfrac{r_+}{q_n(r_+-r_-)}.
\end{equation}
For $n=2$, \eqref{eqn:5.19} and \eqref{eqn:5.34} give the conclusion below. For
$n\geq3$, it follows from \eqref{eqn:5.21}, \eqref{eqn:5.28} and \eqref{eqn:5.34}.
It is clear that all $A_a$ are nonzero and mutually Hermitian orthogonal, and
$\sum_{a=1}^{n+1}\|A_a\|^2=1$. Set $\alpha_a:=\|A_a\|$ and denote the resulting
phase functions by $x_a$. Therefore, up to a constant unitary transformation,
\begin{equation}\label{eqn:5.35}
F=(\alpha_1e^{\mathrm{i}x_1},\ldots,
\alpha_{n+1}e^{\mathrm{i}x_{n+1}}).
\end{equation}
Here the phase functions are explicitly given by
\begin{equation}\label{eqn:5.36}
\begin{aligned}
x_a&=\sum_{j=1}^{a-1}\mu_jw_j+(\lambda_a-\mu_a)w_a,
\ \ 1\leq a\leq n-1,\\
x_n&=\sum_{j=1}^{n-1}\mu_jw_j+r_+w_n,\ \
x_{n+1}=\sum_{j=1}^{n-1}\mu_jw_j+r_-w_n.
\end{aligned}
\end{equation}
Note that each complex coordinate has constant modulus $\alpha_a$, as in the
weighted Clifford torus. For every tangent vector $X$, we have $\langle F_*X,
\mathrm{i}F\rangle=\sum_{a=1}^{n+1}\alpha_a^2X(x_a)$. As $\xi=-\mathrm{i}F$,
the Legendrian condition is therefore equivalent to $\sum_{a=1}^{n+1}\alpha_a^2
dx_a=0$, stating that $\sum_{a=1}^{n+1}\alpha_a^2x_a$ is constant. Denote this
constant by $c_0$. Since $\sum_{a=1}^{n+1}\alpha_a^2=1$, replacing each $x_a$
by $x_a-c_0$ shows that $\sum_{a=1}^{n+1}\alpha_a^2(x_a-c_0)=c_0-c_0=0$ and
replaces $F$ by $e^{-\mathrm{i}c_0}F$, a constant unitary transformation. We
may therefore assume that $\sum_{a=1}^{n+1}\alpha_a^2x_a=0$. Since $F$ is an
immersion, \eqref{eqn:5.35} implies that the map $(x_1,\ldots,x_{n+1})$ has
rank $n$ and is a local diffeomorphism onto an open subset of $V_\alpha$.
Hence $F$ is locally congruent to $\Psi_\alpha$ in Example \ref{exa:1.2},
completing the proof of Theorem \ref{thm:5.1}.
\end{proof}

\subsection{Completion of the proof of Theorem \ref{thm:1.1}}\label{sect:5.1}~

Let $c$ denote the constant sectional curvature of $M^n$. Assume first that
$c>0$. Since all terms in \eqref{eqn:4.29} are nonnegative, Lemma \ref{lem:4.2}
gives $\tilde K=S=0$. It then follows from \eqref{eqn:4.8} that $\|T\|$ is constant
and hence the second identity in \eqref{eqn:4.27} reduces to $0=2\rho T^\flat$. If
$T$ were nonzero at some point, then $\rho=0$ on a neighborhood where $T\ne0$. The
first identity in \eqref{eqn:4.27} implies that $0=d\rho=-cT^\flat$, a contradiction.
Thus $T=0$ and $K=0$ by the definition of $\tilde K$, which also gives $\rho=0$. The
Gauss equation yields $c=1$, and $M^n$ is locally a totally geodesic Legendrian sphere.

Let $c=0$ and fix $p\in M^n$. Choose the orthonormal basis $\{e_i\}$ in Lemma
\ref{lem:4.3}, for which \eqref{eqn:4.56} holds at $p$. Denote by $\pi_i$ the
orthogonal projection onto $\operatorname{span}\{e_i\}$:
\begin{equation}\label{eqn:5.37}
\pi_i(X)=g(X,e_i)e_i,\ \  X\in T_pM^n.
\end{equation}
Suppose that $\rho(p)\ne0$. According to the normal form \eqref{eqn:4.56},
it satisfies
\begin{equation}\label{eqn:5.38}
P_{e_ie_i}=n\rho(p)\pi_i,\ \
P_{e_ie_j}=0,\ \  \forall\,i\ne j.
\end{equation}
Using $U=X=e_i$ and $Y=e_j$ in \eqref{eqn:4.53}, we further obtain
\begin{equation}\label{eqn:5.39}
[\pi_i,K_{e_j}]=0,\ \ \forall\,i\ne j.
\end{equation}
Fixing $j$, we conclude that $K_{e_j}$ preserves each
$\operatorname{span}\{e_i\}$ with $i\ne j$, so
$K_{e_j}e_i=a_{ij}e_i$ for some $a_{ij}\in\mathbb R$. Since
$K_{e_j}$ is self-adjoint,
\begin{equation}\label{eqn:5.40}
g(K_{e_j}e_j,e_i)=g(e_j,K_{e_j}e_i)=0,\ \ \forall\,i\ne j.
\end{equation}
Consequently, $K_{e_j}$ is diagonal in this same basis for every $j$. By linearity,
$[K_X,K_Y]=0$ for all $X,Y\in T_pM^n$. This contradicts the Gauss equation $[K_X,K_Y]
=-X\wedge Y$, and hence $\rho=0$. This together with \eqref{eqn:4.48} implies that
$\nabla K=0$ and Theorem \ref{thm:5.1} shows that $M^n$ is locally
congruent to a weighted Clifford Legendrian immersion $\Psi_\alpha$
in Example \ref{exa:1.2}.

Finally, the case $c<0$ is excluded by Proposition \ref{pro:4.3}. Conversely,
the totally geodesic Legendrian sphere has $K=0$, while the weighted Clifford
immersions are flat and have $\nabla K=0$ by Proposition \ref{pro:3.1}. Hence
all the stated submanifolds have conformal Maslov form and constant sectional
curvature. This completes the proof of Theorem \ref{thm:1.1}.\qed

\subsection{Completion of the proof of Theorem \ref{thm:1.2}}\label{sect:5.2}~

The tensor $\tilde K$ is parallel by Proposition \ref{pro:4.4}, since $\sec_g\geq0$.
With additional $\sec_g\leq1$, Lemma \ref{lem:4.6} guarantees that $\nabla K=0$.
Thus the second fundamental form is $C$-parallel by \eqref{eqn:2.28} and Lemma
\ref{lem:4.7} shows that $M^n$ is totally geodesic, equivalently $K=0$, or flat.

If $K=0$, the embedded submanifold $M^n$ is open in a totally geodesic Legendrian
sphere. Compactness also makes it closed, so it is the whole sphere.

If $M^n$ is flat, completeness identifies its universal cover with Euclidean space.
A global parallel frame extends the constant coefficient integration in the proof
of Theorem \ref{thm:5.1} to this cover, giving a single representation $\Psi_\alpha$
up to a constant unitary transformation. Every deck transformation preserves each
exponential coordinate of $\Psi_\alpha$. Its phase displacements are continuous $2
\pi\mathbb Z$-valued functions and hence constant, so it is a translation by a period.
Conversely, since $M^n$ is embedded, two points of the universal cover have the same
image under $\Psi_\alpha$ if and only if they lie in the same deck orbit. Thus every
translation by a period is a deck transformation, and the deck group is exactly
$P_\alpha$. Compactness implies that $P_\alpha$ has rank $n$. By Remark \ref{rem:3.1},
$M^n$ is therefore congruent to the closed embedded torus $F_\gamma(T^n_\gamma)$ in
Example \ref{exa:1.3} for some primitive positive integer vector $\gamma$. Finally,
the totally geodesic sphere has $K=T=0$, while Proposition \ref{pro:3.1} and Remark
\ref{rem:3.1} verify the converse for these embedded tori. Taking the trace of
$\nabla K=0$ gives $\nabla T=0$. This completes the proof of Theorem \ref{thm:1.2}.\qed

\vskip3mm\noindent
{\bf Statements and Declarations}

\vskip2mm
\noindent{\bf Competing interests}
The authors declare no conflicts of interest.

\vskip2mm\noindent
{\bf Data availability}
No datasets were generated or analysed during the current study.



\begin{thebibliography}{99}

\bibitem{AL15}
Andrews, B., Li, H.:
Embedded constant mean curvature tori in the three-sphere.
J. Differential Geom. \textbf{99}, 169--189 (2015)

\bibitem{BBK95}
Baikoussis, C., Blair, D.E., Koufogiorgos, T.:
Integral submanifolds of Sasakian space forms $\overline M^7(k)$.
Results Math. \textbf{27}, 207--226 (1995)

\bibitem{Bes87}
Besse, A.L.:
Einstein Manifolds.
Springer-Verlag, Berlin (1987)

\bibitem{Bla10}
Blair, D.E.:
Riemannian Geometry of Contact and Symplectic Manifolds, 2nd edn.
Birkh\"{a}user, Boston (2010)

\bibitem{But09}
Butscher, A.:
Equivariant gluing constructions of contact stationary Legendrian
submanifolds in $\mathbb S^{2n+1}$.
Calc. Var. Partial Differential Equations \textbf{35}, 57--102 (2009)

\bibitem{Bre13}
Brendle, S.:
Embedded minimal tori in $\mathbb S^3$ and the Lawson conjecture.
Acta Math. \textbf{211}, 177--190 (2013)

\bibitem{Car38}
Cartan, \'{E}.:
Familles de surfaces isoparam\'etriques dans les espaces
\`a courbure constante.
Ann. Mat. Pura Appl. (4) \textbf{17}, 177--191 (1938)

\bibitem{Cas96}
Castro, I.:
Lagrangian surfaces with conformal Maslov form. In: Geometry and Topology of Submanifolds, VIII
(Brussels, 1995/Nordfjordeid, 1995), pp. 46--55. World Sci. Publishing, River Edge, NJ (1996)

\bibitem{CMU01}
Castro, I., Montealegre, C.R., Urbano, F.:
Closed conformal vector fields and Lagrangian submanifolds in complex space forms.
Pacific J. Math. \textbf{199}, 269--302 (2001)

\bibitem{CU93}
Castro, I., Urbano, F.:
Lagrangian surfaces in the complex Euclidean plane with conformal Maslov form.
Tohoku Math. J. (2) \textbf{45}, 565--582 (1993)

\bibitem{CD12}
Chao, X., Dong, Y.:
Rigidity theorems for Lagrangian submanifolds of complex space forms
with conformal Maslov form. In: Recent Developments in Geometry and Analysis,
Adv. Lect. Math. (ALM), vol. 23, pp. 17--25. Int. Press, Somerville, MA (2012)

\bibitem{CHH21}
Cheng, X., He, H., Hu, Z.:
$C$-totally real submanifolds with constant sectional curvature in the
Sasakian space forms.
Results Math. \textbf{76}, Article No. 144, 16 pp. (2021)

\bibitem{CH23}
Cheng, X., Hu, Z.:
On $C$-totally real submanifolds of $\mathbb S^{2n+1}(1)$ with
non-negative sectional curvature.
Kodai Math. J. \textbf{46}, 184--206 (2023)

\bibitem{CHY19}
Cheng, X., Hu, Z., Yao, Z.:
A rigidity theorem for centroaffine Chebyshev hyperovaloids.
Colloq. Math. \textbf{157}, 133--141 (2019)

\bibitem{CLH26}
Cheng, X., Li, Q., Hu, Z.:
Legendrian submanifolds of $\mathbb{S}^{2n+1}(\widetilde c)$ with constant sectional
curvature and conformal Maslov form.
Acta Math. Sin. (Engl. Ser.) \textbf{42}, 1899--1920 (2026)

\bibitem{CDK70}
Chern, S.S., do Carmo, M., Kobayashi, S.:
Minimal submanifolds of a sphere with second fundamental form of constant length.
In: Browder, F.E. (ed.) Functional Analysis and Related Fields,
pp. 59--75. Springer, Berlin (1970)

\bibitem{CCJ07}
Cecil, T.E., Chi, Q.S., Jensen, G.R.:
Isoparametric hypersurfaces with four principal curvatures.
Ann. of Math. (2) \textbf{166}, 1--76 (2007)

\bibitem{DV90}
Dillen, F., Vrancken, L.:
$C$-totally real submanifolds of Sasakian space forms.
J. Math. Pures Appl. (9) \textbf{69}, 85--93 (1990)

\bibitem{HLWZ20}
Hu, Y., Li, H., Wei, Y., Zhou, T.:
Contraction of surfaces in hyperbolic space and in sphere.
Calc. Var. Partial Differential Equations \textbf{59},
Article No. 172, 32 pp. (2020)

\bibitem{HLX22}
Hu, Z., Li, M., Xing, C.:
On $C$-totally real minimal submanifolds of the Sasakian space forms
with parallel Ricci tensor.
Rev. R. Acad. Cienc. Exactas F\'is. Nat. Ser. A Mat. RACSAM \textbf{116},
Article No. 163, 25 pp. (2022)

\bibitem{Imm08}
Immervoll, S.:
On the classification of isoparametric hypersurfaces with four
distinct principal curvatures in spheres.
Ann. of Math. (2) \textbf{168}, 1011--1024 (2008)

\bibitem{Law69}
Lawson, H.B.:
Local rigidity theorems for minimal hypersurfaces.
Ann. of Math. (2) \textbf{89}, 187--197 (1969)

\bibitem{LW01}
L\^e, H.V., Wang, G.:
A characterization of minimal Legendrian submanifolds in $\mathbb S^{2n+1}$.
Compositio Math. \textbf{129}, 87--93 (2001)

\bibitem{LLV20}
Lee, J.W., Lee, C.W., V\^{\i}lcu, G.-E.:
Classification of Casorati ideal Legendrian submanifolds in Sasakian space forms.
J. Geom. Phys. \textbf{155}, Article No. 103768, 13 pp. (2020)

\bibitem{LXY26}
Li, C., Xing, C., Yin, J.:
On conformally flat minimal Legendrian submanifolds in the unit sphere.
Proc. Roy. Soc. Edinburgh Sect. A \textbf{156}, 284--313 (2026)

\bibitem{Luo17}
Luo, Y.:
On Willmore Legendrian surfaces in $\mathbb S^5$ and the contact
stationary Legendrian Willmore surfaces.
Calc. Var. Partial Differential Equations \textbf{56},
Article No. 86, 19 pp. (2017)

\bibitem{Luo18}
Luo, Y.:
Contact stationary Legendrian surfaces in $\mathbb S^5$.
Pacific J. Math. \textbf{293}, 101--120 (2018)

\bibitem{LS22}
Luo, Y., Sun, L.:
Rigidity of closed CSL submanifolds in the unit sphere.
Ann. Inst. H. Poincar\'e C Anal. Non Lin\'eaire {\bf40}, 531--555 (2023)

\bibitem{LSY22}
Luo, Y., Sun, L., Yin, J.:
An optimal pinching theorem of minimal Legendrian submanifolds in the unit sphere.
Calc. Var. Partial Differential Equations {\bf61}, Paper No. 192, 18 pp. (2022)

\bibitem{Mih17}
Mihai, I.:
On the generalized Wintgen inequality for Legendrian submanifolds
in Sasakian space forms.
Tohoku Math. J. (2) \textbf{69}, 43--53 (2017)

\bibitem{MN14}
Marques, F.C., Neves, A.:
Min-max theory and the Willmore conjecture.
Ann. of Math. (2) \textbf{179}, 683--782 (2014)

\bibitem{Pit05}
Piti\c{s}, G.:
Integral submanifolds with closed conformal vector field in Sasakian manifolds.
New York J. Math. \textbf{11}, 157--170 (2005)

\bibitem{Ros85}
Ros, A.:
A characterization of seven compact Kaehler submanifolds by
holomorphic pinching.
Ann. of Math. (2) \textbf{121}, 377--382 (1985)

\bibitem{Sas14}
Sasahara, T.:
A class of biminimal Legendrian submanifolds in Sasakian space forms.
Math. Nachr. \textbf{287}, 79--90 (2014)

\bibitem{Sim68}
Simons, J.:
Minimal varieties in Riemannian manifolds.
Ann. of Math. (2) \textbf{88}, 62--105 (1968)

\bibitem{SS20}
Sun, J., Sun, L.:
Sphere theorems for Lagrangian and Legendrian submanifolds.
Calc. Var. Partial Differential Equations \textbf{59},
Article No. 125, 29 pp. (2020)

\bibitem{Urb86}
Urbano, F.:
Nonnegatively curved totally real submanifolds.
Math. Ann. \textbf{273}, 345--348 (1986)

\bibitem{WX26}
Wang, X., Xing, C.:
Conformally flat Lagrangian submanifolds of complex projective space
with parallel mean curvature vector.
Submitted

\bibitem{YKI76}
Yamaguchi, S., Kon, M., Ikawa, T.:
$C$-totally real submanifolds.
J. Differential Geom. \textbf{11}, 59--64 (1976)

\bibitem{Yan57}
Yano, K.:
The Theory of Lie Derivatives and its Applications.
North-Holland, Amsterdam (1957)

\bibitem{YQ22}
Yin, J., Qi, X.:
Sharp estimates for the first eigenvalue of Schr\"odinger operator in the unit sphere.
Proc. Amer. Math. Soc. \textbf{150}, 3087--3101 (2022)

\end{thebibliography}
\end{document}